\documentclass [12pt,a4paper,reqno]{amsart}
\usepackage{etex}
\usepackage{amsmath, amssymb}
\usepackage{extarrows}
\usepackage{xypic}

\usepackage{amsmath,amsthm}
\usepackage{amssymb,amscd}
\usepackage{mathabx}
\usepackage{latexsym}
\usepackage{graphicx}

\usepackage{pstricks}

\usepackage{tikz}
\usetikzlibrary{positioning}

\input xy
\xyoption{all}

\newcommand{\etype}[1]{\renewcommand{\labelenumi}{(#1{enumi})}}

\def\ealph{\etype{\alph} \dispace }
\def\dispace{\setlength{\itemsep}{2pt}}
\newcommand{\ds}[1]{\ {#1} \ }
\newcommand{\dss}[1]{\quad {#1} \quad }

\def\Iff{\ \Leftrightarrow \ }
\def\udscr{-}

\def\chP{\widecheck{P}}
\def\chT{\widecheck{T}}
\def\chX{\widecheck{X}}

\def\sm{\setminus}
\def\00{ \{ 0 \}}
\def\simr{\sim_{\operatorname{r}}}

\def\Min{{\operatorname{Min}}}

\def\Yup{Y^\uparrow}
\def\Zup{Z^\uparrow}

\newcommand{\Qdefref}[1]{\cite[Definition~{#1}]{QF1}}

\def\tlA{\widetilde{A}}
\def\tlB{\widetilde{B}}
\def\tlC{\widetilde{C}}
\def\tlT{\widetilde{T}}

\def\tlX{\widetilde{X}}

\def\sm{\setminus}

\def\nucong{\cong_\nu}
\def\noi{\noindent}

\def\pSkip{\vskip 1.5mm \noindent}
\def\semiring0{semiring$^{\dagger}$}

\newcommand{\conv}{\operatorname{conv}}

\newcommand{\Ray}{\operatorname{Ray}}
\newcommand{\ray}{\operatorname{ray}}
\newcommand{\CS}{\operatorname{CS}}
\newcommand{\an}{{\operatorname{an}}}
\newcommand{\QL}{\operatorname{QL}}

\newcommand{\Max}{\rm Max}
\newcommand{\Asc}{\rm Asc}
\newcommand{\Loc}{\rm Loc}
\newcommand{\glen}{\rm glen}
\newcommand{\Glen}{\rm Glen}

\def\Max{\mathop{\rm Max}\limits}

\def\tT{\mathcal T}

\def\tG{\mathcal G}

\newtheorem{thm}{Theorem} [section]
\newtheorem*{thm*}{Theorem}
\newtheorem{cor}[thm]{Corollary}
\newtheorem{lem}[thm]{Lemma}

\newtheorem{prop}[thm]{Proposition}

\newtheorem*{claim*} {Claim}

\newtheorem*{theorem13.5'} {Theorem 13.5$'$}
\newtheorem{acknowledgment*}[thm] {Acknowledgment}

\newtheorem{examp}[thm]{Example}

  \newtheorem{remarks}[thm]{Remarks}
    \newtheorem*{remarks*} {Remarks}

 \newtheorem*{remark*}{Remark}
 \newtheorem{defn}[thm]{Definition}

\newtheorem{schol}[thm]{Scholium}
\newtheorem{notation}[thm]{Notation}
\newtheorem{problem}[thm]{Problem}

\newtheorem*{notation*} {Notation}

\newtheorem{rem}[thm]{Remark}
\newtheorem{tabl}[thm]{Table}

\newcommand{\st}{\operatorname{st}}

 \renewcommand{\sectionmark}[1]{}

\newcommand{\bfem}[1]{\textbf{#1}}

\newcommand{\veps}{\varepsilon}
\newcommand{\al}{\alpha}
\newcommand{\bt}{\beta}
\newcommand{\gm}{\gamma}

\newcommand{\lm}{\lambda}
\newcommand{\Lm}{\Lambda}
\newcommand{\om}{\omega}

\begin{document}

\title[Cauchy-Schwarz functions]{Cauchy-Schwarz functions    and convex partitions  \\[2mm] in the ray space \\[2mm]of a supertropical quadratic form} \author[Z. Izhakian]{Zur Izhakian}
\address{Institute  of Mathematics,
 University of Aberdeen, AB24 3UE,
Aberdeen,  UK.}
    \email{zzur@abdn.ac.uk}
\author[M. Knebusch]{Manfred Knebusch}
\address{Department of Mathematics,
NWF-I Mathematik, Universit\"at Regensburg 93040 Regensburg,
Germany} \email{manfred.knebusch@mathematik.uni-regensburg.de}

\date{\today}



\begin{abstract}
Rays are  classes of an equivalence relation on a module $V$ over a supertropical semiring.  They provide a version of  convex geometry,  supported by a ``supertropical trigonometry'' and compatible with quasilinearity, in which the CS-ratio takes the  role of the Cauchy-Schwarz inequality.  CS-functions which emerge from the CS-ratio are a useful tool that helps to understand the variety of quasilinear stars in the ray space $\Ray(V)$. In particular, these functions induce a partition of $\Ray(V)$ into convex sets, and thereby a finer convex analysis which includes the notions of median, minima,  glens, and polars.

\end{abstract}

\maketitle
\setcounter{tocdepth}{1}
{ \small \tableofcontents}

\numberwithin{equation}{section}
\section{Introduction}

Quadratic forms on a free supertropical module, and  their bilinear companions, were introduced and classified in \cite{QF1,QF2}, and studied further in \cite{Quasilinear,VR1,UB}.
These objects  establish   a   version of tropical trigonometry, where the CS-ratio takes the  role of the Cauchy-Schwarz inequality, which is  not always applicable. ({``CS'' is an acronym of
``Cauchy-Schwarz''.}) With the notion of CS-ratio, the space of equivalence classes of a suitable equivalence relation, termed rays, provides a framework which carries a type of convex geometry. The study of this geometry was initiated in  \cite{Quasilinear}, focusing on the so called \emph{quasilinear stars}. The present paper proceeds to develop this theory, employing mostly special characteristic functions, called CS-functions,  that emerge from the CS-ratio on ray spaces. These  CS-functions provide a useful tool for convex analysis, which is of much help in  understanding the variety of quasilinear stars in the ray space.

Supertropical modules are modules over supertropical semirings, which carry  a rich algebraic structure   \cite{zur05TropicalAlgebra,nualg,IzhakianKnebuschRowen2010Linear,IR1,IR2,IR3}, and are at the heart of our framework.
A supertropical semiring (\Qdefref{0.3}) is  a semiring $R$ with idempotent element $e:=1+1$
(i.e., $e+e = e$) such that,   for
all~ $a,b\in R$,  $a+b\in\{a,b\}$ whenever $ea \ne eb$ and  $a+b=ea$ otherwise.
The ideal $eR$ of~ $R$ is a bipotent semiring (with unit element $e$),  i.e., $a+b$
is either $a$ or $b$, for any $a,b\in eR$. The total
ordering
\begin{equation*}\label{eq:0.5}
a\le b \dss \Leftrightarrow a+b=b
\end{equation*}
of $eR$, together with the ghost  map
$\nu: a\mapsto ea$, induces the $\nu$-ordering
\begin{equation}\label{eq:nuorderring}
a <_\nu b \dss \Iff ea < eb \end{equation}
 and the $\nu$-equivalence
\begin{equation}\label{eq:nucong}
a \nucong b \dss\Iff ea = eb
\end{equation}
on the entire semiring  $R$, which  determines the addition of~$R$:
\begin{equation*}\label{eq:0.6}
a+b =\begin{cases} b&\ \text{if}\ a <_\nu b,\\
a&\ \text{if}\ a>_\nu  b,\\
eb&\ \text{if}\ a \nucong b.
\end{cases}
\end{equation*}
Consequently, $ea=0  \Rightarrow a =0$, and the
zero $0 = e0$ is regarded mainly as a ghost. The set
$\tT:=R\sm(eR)$ consists of the  \bfem{tangible} elements of $R$, while the ideal $\tG :=(eR)\sm\{0\}$ contains  the \bfem{ghost} elements. 
The semiring $R$ itself is said to be  \bfem{tangible}, if $e\tT=\tG$, i.e., $R$ is generated by $\tT$ as
a semiring.
Then, for $\tT \neq \emptyset$,  $R':=\tT \cup e\tT \cup\{0\}$
is the largest tangible sub-semiring of $R$.

An $R$-module $V$ over a commutative supertropical semiring $R$  is defined in the familiar way.
 A \bfem{quadratic form} on $V$ is a function
$q: V\to R$ satisfying
  \begin{equation*}
  q(ax)=a^2q(x)
 \end{equation*}
for any $a\in R$, $x\in V$, for which there exists  a symmetric bilinear form $b:V\times V\to R$,  called a \bfem{companion} of $q$, such that
    \begin{equation*}
  q(x+y)= q(x)+q(y)+b(x,y)
 \end{equation*}
for any $x,y\in V$.
  ($q$ may have  several companions.)
 The pair $(q,b)$ is called a \bfem{quadratic pair}. It  is called \textbf{balanced}, and $b$ is said to be
a \textbf{balanced companion} of $q$, if $b(x, x) = e q(x)$ for any $x \in V$.

In our version of ``\textit{tropical trigonometry}''
the familiar formula
$ \cos(x,y) = \frac {\langle x,y\rangle} {\| x\| \; \| y\| } $
in euclidian geometry is replaced by
the CS-ratio
\begin{equation}\label{eq:0.10}
\CS(x,y):=\frac{eb(x,y)^2}{eq(x)q(y)}\in eR
\end{equation}
of  \textbf{anisotropic} vectors $x,y \in V$, i.e., $q(x)\ne0,$
$q(y)\ne 0$.
(As for any supertropical semiring  the map $\lm \mapsto \lm^2$  is an injective endomorphism, there is no loss of information by squaring $\CS(x,y)$ \cite[Proposition 0.5]{QF1}.)
The function $x\mapsto
\CS(x,w)$ is subadditive for any anisotropic vector $w \in V$   (Theorem \ref{thm:II.2.10}).

In this setting, features of noneuclidian geometry arise, since, not like in euclidian geometry, the CS-ratio
$\CS(x,y)$ may take values larger than $e $.
These features are closely related to excessiveness \cite[Definition 2.8]{QF2}. 
When $eR$ is densely ordered, a pair
$(x,y)$ is excessive if $\CS(x,y)>e.$
 When $eR$ is discrete, $(x,y)$ is excessive if either $\CS(x,y)>c_0,$
with $c_0$ the smallest element of $eR$ larger than $e,$  or
$\CS(x,y)=c_0$ and  $q(x)$ or $q(y)$ is
tangible. A pair $(x,y)$ is \textbf{exotic quasilinear}, if $\CS(x,y) = c_0$ and  both $q(x)$ and $q(y)$ are
ghost \cite[~Theorems ~2.7 and 2.14]{QF2}.

A pair of vectors $(x,y)$ is called  \textbf{$\nu$-excessive} (resp. \textbf{$\nu$-quasilinear}), if the pair $(ex,ey)$ is excessive (resp. quasilinear). Since $\CS(x,y) = \CS(ex,ey)$, it is often simpler  to work with $\nu$-excessiveness  and $\nu$-quasilinearity. To wit, $(x,y)$ is $\nu$-quasilinear, if
$$ q(x,y) \cong_\nu q(x)+q(y),$$
and is $\nu$-excessive otherwise. When $eR$ is dense,  $(x,y)$ is  $\nu$-quasilinear iff $\CS(x,y) \leq e$, while for $eR$ discrete,  $\nu$-quasilinear iff $\CS(x,y) \leq c_0$; it is exotic quasilinear iff   $\CS(x,y) = c_0$.

The CS-ratio obeys important
 subadditivity rules, involving $\nu$-excessiveness as well as  $\nu$-quasilinearity, which are utilized in this paper.
\begin{thm}[{\cite[Subadditivity Theorem 3.6]{QF2}}]\label{thm:II.2.10}
Let $x,y,w \in V $  be anisotropic vectors.

\begin{enumerate}
\item[a)]
$ \CS(x+y,w)\ds\le\CS(x,w)+\CS(y,w).$

\item[b)] If $(x,y)$ is  $\nu$-excessive
 and  $\CS(x,w)+\CS(y,w)\ne0,$ then
$$ \CS(x+y,w)\ds <\CS(x,w)+\CS(y,w).$$

\item[c)] If $(x,y)$ is  $\nu$-quasilinear and  either $q(x)\CS(y,w)=q(y)\CS(x,w)$, or
$\CS(x,w)=\CS(y,w)$,  or $q(x)\cong_\nu q(y),$ then
$$\CS(x+y,w)\ds=\CS(x,w)+\CS(y,w).$$
\end{enumerate}
\end{thm}

On an $R$-module  $V$ we use the equivalence relation:  $x \sim y$  iff $\lm x = \mu y$ for some $\lm, \mu \in R \sm \00$ (where~ $\lm, \mu$ need not be invertible as in the usual projective equivalence), whose classes $X$ are called \textbf{rays}. It delivers a  projective version of the theory  on $V \sm \00$, cf.    \cite[\S6]{QF2}.  When ~ $x$ and
$y$ are anisotropic, the CS-ratio $\CS(x,y)$ depends only on the rays $X,Y$ containing $x,y$,  and provides   a well defined CS-ratio $\CS(X,Y)$ for
anisotropic rays $X,Y$, i.e., rays $X,Y$ in $V \sm q^{-1}(0).$

Subadditivity  of rays
occurs on  \textbf{intervals}  $[X,Y]$ with endpoints $X,Y$, cf \S\ref{sec:1}, as a  consequence of Theorem~ \ref{thm:II.2.10}.
A comparison of $\CS(Z,W)$ to $\CS(X,W)+\CS(Y,W)$ for anisotropic ray $Z \in [X,Y]$ and arbitrary $W$ is given by
\cite[Theorem ~7.7]{QF2}, and uniqueness of the boundary of~
$[X,Y]$ by Theorem \ref{thm:II.8.8}.
The \textbf{ray space}  $\Ray(V)$  of $V $  consists of all rays and carries a natural notion of convexity: A subset $A \subset \Ray(V)$ is \textbf{convex}, if $[X,Y] \subset A$ for any $X,Y \in A$.
Basics structures of rays and convexity in $\Ray(V)$  are reviewed in \S\ref{sec:1}.

By relying on a  fine  detailed analysis of the monotonicity behavior of the CS-functions $$\CS(W,-): \Ray(V) \longrightarrow eR$$ on a fixed interval $[Y_1, Y_2]$ in $\Ray(V)$, given in \S\ref{sec:6}, CS-profiles on interval are defined in \S\ref{sec:7}. This fine analysis  enlarges the scope of results in \cite{Quasilinear} and determines a partition of $\Ray(V)$ into convex subsets (Theorem \ref{thm:8.6}) according to the  monotonicity behavior of $\CS(W,-)$ on the intervals $[Y_i, Y_j]$, $1 \leq i < j \leq m$, for a given finite set of rays $\{ Y_1, \dots, Y_m\}$ and  $W$ running through $\Ray(V)$.

A {pair $(X,Y)$ of rays} in $V$ is  \textbf{quasilinear} (with respect  to $q$), if the restriction
 $q |_{ Rx + Ry}$ is quasilinear for any $x \in X$, $y \in Y$. A {subset} $C \subset \Ray(V)$ is \textbf{quasilinear}, if all pairs $(X,Y) $ in $C$ are quasilinear.
 Quasilinearity is governed by \textbf{QL-stars} $\QL(X)$ of rays $X$.  $\QL(X)$ is the set of all $Y \in \Ray(V)$ for which the pair $(X,Y)$ is quasilinear\footnote{$\QL(X)$ is not necessarily quasilienar.}; equivalently, the interval $[X,Y]$ is quasilinear.
   \S\ref{sec:9} presents the \textbf{downset} of a QL-star, this is the set of all QL-stars contained in $\QL(X)$, while \S\ref{sec:10} introduces the \textbf{median} on an interval and links it to convexity properties (Corollary \ref{cor:10.6}).

   The study of median, leads in \S\ref{sec:11} to enquire the existence of extrema  of CS-functions $\CS(W,-)$. Theorem \ref{thm:11.1} specifies a condition and the place where  an $eR$-valued function has a minimal, while  Theorem \ref{thm:11.4} provides an upper bound in terms of generators for CS-functions over finitely generated $eR$-modules.

   Inquiring after the   minima of a CS-function is then a natural question.
  An intriguing issue is that the minimum of $\CS(W,-)$ over the convex hull $\conv(Y_1, \dots , Y_n)$ of $\{Y_1, \dots, Y_n \} $ can be smaller than $\min\limits_{i} \CS(W,Y_i)$.
  The rays for which this holds compose the ``\textbf{glen}'' of $Y_1, \dots, Y_n  $,  which is  discussed in detail in  \S\ref{sec:12}. Glens extend to intervals, and establish  a useful correspondence to CS-functions  (Theorem ~ \ref{thm:12.3}).

Given a finitely generated convex set $C$ in a ray space  $\Ray(V)$, we may ask  whether there exists a quadratic pair $(q,b)$ on $V$ with $q$ anisotropic on $V$. In case that $(q,b)$ exists  we can move a ray~ $W$ around,  examine the minima of $\CS(W,-)$ on $C$. In the easiest case that $q$ is quasilinear on $C$, the following holds.  Every ray $Z$ which is an isolated minimum of a CS-function $\CS(W,-)$, i.e., this function is not constant on an interval emanating from~ $Z$, is an ``indispensable generator'' of $C$, i.e., $Z$ occurs in every set of generators of $C$.
If  $C= \conv( S)$ is the convex hull of some finite set $S$ for which certain pairs $(Z,Z')$ in $S$ are $\nu$-excessive, then the situation is more involved, since $\CS(W,-)$  can be non monotonic on $[ Z,Z']$ and the minimum can be attained  at the $W$-median. Nevertheless,  this gives a constraint, on, say, the minimal sets of generators of $C$. An  intriguing phenomenon  is that one can choose ~$W$ nearly arbitrarily.

Motivated by this phenomenon,
\S\ref{sec:13}  explores the set $\Min\CS(W,C)$ of minima of a CS-function $\CS(W,-)$ on the convex hull $C$ of a finite set $S$ of rays in $\Ray(V)$. Theorem ~\ref{thm:13.1} characterizes properties of $\Min\CS(W,C)$, linking these properties to medians.
A $Z$-polar of a subset ~ $P \subset \Zup = \{ Y \ds | \CS(W,Y) > \CS(W,Z) \} $  is a set of rays (Definition \ref{def:13.6}) for which there exist $X \in P$ with $Y \in \Zup $  such that  $M_W(X,Y) =Z$. This subset is  closed for taking convex hull (Theorem ~\ref{thm:13.9}), compatible with the ordering induced by $Z$ (Theorem~ ~\ref{thm:13.10}), and induces the $Z$-equivalence relation on $\Zup$. The  next step of this study is then
 to describe the  classes of this equivalence relation (Problem \ref{prob:13.14}). In this paper we give only a partial description in terms of convex hulls (Theorem \ref{thm:13.13}), but, by introducing the notion of \textbf{median stars} in \S\ref{sec:14}, we lay out a possible machinery  to address this problem.

\section{Convex sets in the ray space}\label{sec:1} We review our setup as was laid out in \cite{QF2,Quasilinear} in which  $V$ denotes an $R$-module, where $R$ a supertropical semiring $R$ such that $eR$ is a (bipotent) semifield and
 $R \sm \00 $ is closed for  multiplication, i.e.,
$\lm \mu =0  \ds \Rightarrow \lm =0 \ds{\text{or}} \mu =0, \ \text{for any } \lm, \mu \in R .$
$V$ is assumed to have  the property
 $\lm x = 0 \ds \Rightarrow \lm =0 \ds{\text{or}} x =0, \ \text{for any  } x \in V. $
These properties hold when $eR = \tG \cup \00 $ is a semifield.

Vectors $x,y \in V$ are \textbf{ray-equivalent}, written $x \simr y ,$ if  $\lm x = \mu y$ for some  $\lm, \mu \in R \sm \00$.  This is the finest equivalence relation $E$ on $V$ with $x \sim_E \lm x $ for any $\lm \in R \sm \00$,  
which gives $V \sm \00 $ as a union of ray-equivalence classes.
The \textbf{rays} in $V$ are the ray-equivalence classes  $\neq \00$.
The \textbf{ray}  $\ray_V(x)$ of $x$ in $V$ is the ray-equivalence class of a vector $x \in V \sm \00$, written   $\ray(x)$  when $V$ is clear form the context. The \textbf{ray space} $\Ray(V)$ of $V$  is the set of all rays in~ $V$.  The set $X_0 := X \cup \00 $ is the smallest submodule of~$V$ containing the ray $X$.  $X_0 + Y_0$ is the smallest submodule of $V$ containing both rays $X$ and $Y$. It  is a disjoint union of subsemigroups of $(V,+)$ as follows
\begin{equation}\label{eq:II.5.3}
X_0 + Y_0 = (X + Y) \ds{ \cup} X \ds{ \cup} Y \ds{ \cup} \00. \end{equation}
The \bfem{closed interval} $[X,Y]$ consists  of all
rays $Z$ in the submodule $X_0+Y_0$ of $V$, generated by $X\cup Y$. The \bfem{open interval} $\, ]X,Y[$
consists  of all rays
$$Z\subset  X+Y:=\{x+y \ds | x\in X,y\in Y\}.$$
Thus, $[X,Y] \ds =\, ]X,Y[ \ds \cup\{X,Y\}.$
The \textbf{half open intervals} are
$$[X,Y[\ds{:=} ]X,Y[\ds \cup\{X\},\qquad ]X,Y] \ds{:=} ]X,Y[ \ds \cup\{Y\}.$$

\begin{schol}[{\cite[Scholium 7.6]{QF2}}]\label{schol:II.6.6}

Set $X=\ray(x),$ $Y=\ray(y)$ for  $x,y\in V$. For any ray~ $Z$ in
$V,$ the following hold:

\begin{enumerate} \dispace
\item[a)] $Z \ds\in\, ]X,Y[$ iff $Z=\ray(\lm x+\mu y)$ with
$\lm,\mu\in R\setminus\{0\}$ iff $Z=\ray(\lm x+ y)$
with $\lm\in R\setminus\{0\}.$ \pSkip

\item[b)] $Z\ds\in ]X,Y]$ iff
$Z=\ray(\lm x+\mu y)$ with $\lm\in R,$ $\mu\in
R\setminus\{0\}$ iff $Z=\ray(\lm x+y)$ with $\lm\in R.$\pSkip

\item[c)] $Z\in[X,Y]$ iff $Z=\ray(\lm x+\mu y)$ with
$\lm,\mu\in R.$\end{enumerate}
 $R$ and $R\setminus\{0\}$   may be replaced respectively by $eR$
and $\tG$ everywhere.

\end{schol}

   \begin{thm}[{\cite[Theorem 8.8]{QF2}}]\label{thm:II.8.8} Let $X, Y , X_1, Y_1$ be rays in $V$ with $[X,Y] = [X_1, Y_1]$.
    \begin{enumerate}
 \item[a)] Either $X= X_1,$ $Y = Y_1$  or $X = Y_1, $ $Y = X_1.$  \pSkip

 \item[b)] If  $[X,Y]$ is not a singleton, i.e., $X \neq Y,$  then
$X= X_1,$ $Y = Y_1$ iff $[X, X_1] \neq [X, Y_1].$
 \end{enumerate}

\end{thm}

  A subset $M \subset \Ray(V)$ is \textbf{convex} (in $\Ray (V)$), if for any two rays $X, Y \in M$ the closed interval $[X, Y]$ is contained in $M$.
The \textbf{convex hull} $\conv (S)$ of a nonempty set $S \subset \Ray (V)$  is
the smallest  convex subset of $\Ray (V)$ containing $S$.
When  $S = \{ X_1, \dots, X_n \}$ is finite, $ \conv (S) $ is written
$  \conv (X_1, \dots, X_n) $,
for short. Clearly $[X, Y] = \conv (X, Y)$, and by  \cite[Proposition 8.1]{QF2} all the intervals $] X, Y [ \, , ] X, Y ], [ X, Y [ \, , [ X, Y ]$
are convex sets, for any rays $X, Y$ in $\Ray(V)$.
A subset $U \subset V$ is \textbf{\emph{ray-closed}} in $V$, if $U \sm \00$ is a
union of rays of~ $V$. 

\begin{prop}[{\cite[Proposition 2.6]{Quasilinear}}]\label{prop:1.3}

$ $ \begin{enumerate} \ealph
      \item
     If $U_1, \dots, U_n$ are ray-closed subsets of $V \setminus \{ 0 \}$, 
     then the set $U_1 + \dots + U_n$ is again ray-closed in $V$, consisting of all rays $\ray_V (\lm_1 u_1 + \dots + \lm_n u_n)$ with $u_i \in U_i$, $\lm_i \in R \setminus \{ 0 \}$. In particular, for any rays $X_1, \dots, X_n$ in $V$ the set $X_1 + \dots + X_n$ is ray-closed in $V$.
\item The convex hull of a finite set of rays $\{ X_1, \dots, X_n \}$ has the disjoint decomposition
\[ \conv (X_1, \dots, X_n) = \bigcup\limits_{i_1 < \dots < i_r} \Ray (X_{i_1} + \dots + X_{i_r}) \]
with $r \leq n$,  $1 \leq i_1 < \dots < i_r \leq n$.
\end{enumerate}

\end{prop}

We denote by $\tlA$ the set of all sums of finitely many members of  a subset  $A \subset \Ray (V)$.
As the convex hull of  $A$  is the union of all sets $\conv (X_1, \dots, X_r)$ with $r \in \mathbb{N}$, $X_1, \dots, X_r \in A$,   we obtain the following.

\begin{cor}[{\cite[Corollary 2.8]{Quasilinear}}]\label{cor:1.5} Let $C$ denote the convex hull of $A_1 \cup \dots \cup A_n$, where $A_1, \dots, A_n$ are convex subsets of $\Ray (V)$.
\begin{enumerate}
  \ealph
\item $C$ is the union of all convex  hulls $\conv (X_1, \dots, X_n)$ with $X_i \in A_i$,  $1 \leq i \leq n$.
\item $\tlC$ is the union of all sets $\tlA_{i_1} + \dots + \tlA_{i_r}$ with $r \leq n$, $1 \leq i_1 < \dots < i_r \leq n$.
\end{enumerate}
\end{cor}
Proposition \ref{prop:1.3} and  Corollary \ref{cor:1.5} can be inferred from the next observation.

\begin{prop}[{\cite[Proposition 2.9]{Quasilinear}}]\label{prop:1.6}
The convex subsets $A$ of $\Ray (V)$ correspond uniquely to the ray-closed submodules $W$ of $V$ via
$ W = \tlA \cup \{ 0 \} ,  A = \Ray (W). $
\end{prop}

\section{The function $\CS (X_1, - )$ on $[X_2, X_3]$}\label{sec:6}

In  this section $R$ denotes  a supertropical semiring  whose ghost ideal $eR = \{ 0 \} \cup \tG$ is a nontrivial (bipotent) semifield,   and  $V$ stands for  an $R$-module equipped with a fixed quadratic pair $(q, b)$ with $q$ anisotropic on $V$, i.e., $q^{-1}(0) = \{0\}$. Assuming that $X_1, X_2, X_3$ are three  rays on ~$V$,  we explicitly  analyze the monotonicity behavior of the function $\CS (X_1, \udscr)$ on the interval $[X_2, X_3]$ with  $X_2 \ne X_3$. For vectors $\veps_i \in X_i$, $i = 1, 2, 3$,  we employ the following six parameters
\[ \al_i = q (\veps_i) \ne 0, \quad \al_{ij} = b (\veps_i, \veps_j) = \al_{ji} \qquad i, j \in \{ 1, 2, 3 \}, \ i < j. \]
 As our computations take place in the semifield $eR$, we can write $\leq, <$ instead of $\leq_{\nu}, <_{\nu}$. But, the forthcoming formulas are to be used later in a supertropical context without fuss, otherwise we could assume that the parameters $\al_i, \al_{ij}$ belong to $eR$.

Our analysis is performed in terms of the function
\[ f (\lm) = \CS (\veps_1, \veps_2 + \lm \veps_3) \]
with $\lm \in eR \cup \{ \infty \} = \{ 0 \} \cup \tG \cup \{ \infty \}$. Here $\lm = \infty$ corresponds to $\mu = 0$ for $\mu = \lm^{-1}$, and $\CS (\veps_1, \veps_2 + \lm \veps_3) = \CS (\veps_1, \mu \veps_2 + \veps_3)$, thus  $f(\infty) = \CS (\veps_1, \veps_3)$.
We have
\[ b (\veps_1, \veps_2 + \lm \veps_3) = \al_{12} + \lm \al_{13}, \]
and so
\begin{equation}\label{eq:6.2} f (\lm) = e \frac{\al_{12}^2 + \lm^2 \al_{13}^2}{\al_1 \; q (\veps_2 + \lm \veps_3)} \in eR,\end{equation}
which decomposes as
\begin{equation}\label{eq:6.3} f (\lm) = f_1 (\lm) + f_2 (\lm) = \max (f_1 (\lm), f_2 (\lm)) \end{equation}
with \begin{equation}\label{eq:6.4} {\displaystyle f_1 (\lm) = e \frac{\al_{12}^2}{\al_1 \; q (\veps_2 + \lm \veps_3)} , \qquad  f_2 (\lm) = e \frac{\lm^2 \al_{13}^2}{\al_1 \; q (\veps_2 + \lm \veps_3)}}.\footnote{In the following we omit the factor $e = 1_{eR}$, reading all formulas in $eR$.}
\end{equation}
We proceed by analysing the monotonicity behavior of $f_1, f_2$ on $[0, \infty]$. Without loss of generality we assume that\begin{equation}\label{eq:6.5} \CS (\veps_1, \veps_2) \leq \CS (\veps_1, \veps_3). \end{equation}
(Otherwise interchange $X_2$ and $X_3$.)

If $\CS (\veps_1, \veps_3) = 0$, then $\al_{12} = \al_{13} = 0$, and thus $f_1 = 0$, $f_2 = 0$, $f = 0$. \textit{Discarding this trivial case,  we assume that} $\CS (\veps_1, \veps_3) > 0$.
We rewrite the functions $f_1, f_2$ as follows
\begin{equation}\label{eq:6.6} f_1 (\lm) = \frac{\al_{12}^2}{\al_1 \al_2} \; \frac{\al_2}{q (\veps_2 + \lm \veps_3)} = \CS (\veps_1, \veps_2) \frac{q (\veps_2)}{q (\veps_2 + \lm \veps_3)}, \end{equation}
\begin{equation}\label{eq:6.7} f_2 (\lm) = \frac{\al_{13}^2}{\al_1 \al_3} \; \frac{\lm^2 \al_3}{q (\veps_2 + \lm \veps_3)} = \CS (\veps_1, \veps_3) \frac{q (\lm \veps_3)}{q (\veps_2 + \lm \veps_3)}. \end{equation}
These formulas imply that
\begin{equation}\label{eq:6.8} f_1 (\lm) \leq \CS (\veps_1, \veps_2), \quad  \mbox{$f$}_2 (\lm) \leq \CS (\veps_1, \veps_3), \qquad \text{for all } \lm \in [0, \infty]. \end{equation}
 More explicitly,
\begin{equation}\label{eq:6.9} q (\veps_2 + \lm \veps_3) = \al_2 + \lm \al_{23} + \lm^2 \al_3, \end{equation}
and so
\begin{equation}\label{eq:6.10} f_1 (\lm) = \CS (\veps_1, \veps_2) \frac{\al_2}{\al_2 + \lm \al_{23} + \lm^2 \al_3}, \end{equation}
\begin{equation}\label{eq:6.11}f_2 (\lm) = \CS (\veps_1, \veps_3) \frac{\al_3}{\al_3 + \lm^{-1} \al_{23} + \lm^{-2} \al_2}. \end{equation}
We conclude from these formulas that $f_1$ decreases (monotonically) on $[0, \infty]$  from $\CS (\veps_1, \veps_2)$ to zero, while $f_2$ increases (monotonically) from zero to $\CS (\veps_1, \veps_3)$. Moreover, we infer from~ \eqref{eq:6.4} that $f_1 (\lm) = f_2 (\lm)$ precisely when $\al^2_{12} = \lm^2 \al^2_{13}$, whence $\lm^2 = \frac{\al_{12}^2}{\al_{13}^2}$. So $f_1 (\lm) = f_2 (\lm)$ holds on the unique argument $\xi$, that is
\begin{equation}\label{eq:6.12} \xi = \frac{\al_{12}}{\al_{13}}. \end{equation}
It follows that $f$ coincides with $f_1$ on $[0, \xi]$ and with $f_2$ on $[\xi, \infty]$. Furthermore $f(\xi) = f_1 (\xi) = f_2 (\xi)$ is the minimal value attained by the function $f$ on $[0, \infty]$. In other words, $f(\xi)$ is the minimal value of $\CS (X_1, Z)$ for $Z$ running over  $[X_2, X_3]$.
$\xi$ corresponds to the ray
\begin{equation}\label{eq:6.13} M : = \ray \bigg(\veps_2 + \frac{\al_{12}}{\al_{13}} \veps_3\bigg) = \ray (\al_{13} \veps_2 + \al_{12} \veps_3),  \end{equation}
which we call the $X_1$-\textbf{median} of the interval $[X_2, X_3]$. This important ray $M$ will be studied in detail later.

So far we have obtained an outline of the monotonicity behavior of $f_1, f_2, f$. This picture will now be refined. We  start with the case that
\begin{equation}\label{eq:6.14} \al_2 \al_3 \leq \al_{23}^2, \end{equation}
which, except in the border case $\al_2 \al_3 = \al_{23}^2$, implies that the interval $[X_2, X_3]$ is excessive or exotic quasilinear \cite[Definition 2.8]{QF2}. In particular $\al_{23} \ne 0$.

We determine the subsets of $[0, \infty]$ where the decreasing function $f_1$ takes its maximal value $\CS (\veps_1, \veps_2)$ and the increasing function $f_2$ takes its maximal value $\CS (\veps_1, \veps_3)$ as follows.
When $\CS (\veps_1, \veps_2) \ne 0$, we read off from Formula  \eqref{eq:6.8}, applied to  $q (\veps_2 + \lm \veps_3)$, that $f_1 (\lm) = \CS (\veps_1, \veps_2)$ precisely when the summand $\al_2$  is $\nu$-dominant, i.e., $\al_2 \geq \lm \al_{23}$, $\al_2 \geq \lm^2 \al_3$, equivalently,
\[ \lm^2 \leq \frac{\al_2^2}{\al_{23}^2}, \qquad \lm^2 \leq \frac{\al_2}{\al_3}. \]
From \eqref{eq:6.14} we infer that $\frac{\al_2}{\al_3} \geq \frac{\al_2^2}{\al_{23}^2}$, and  therefore   the condition $\lm^2 \leq \frac{\al_2}{\al_3}$ can be dismissed. Thus
\begin{equation}\label{eq:6.15} f_1 (\lm) = \CS (\veps_1, \veps_2) \dss\Iff \lm \leq \frac{\al_2}{\al_{23}}. \end{equation}
The case of $\CS (\veps_1, \veps_2) = 0$  is degenerate, in  which $f_1 = 0$, $f_2 = f$ (and $\xi = 0$).

Concerning $f_2$, we read off from \eqref{eq:6.7} and \eqref{eq:6.11} that $f_2 (\lm) = \CS (\veps_1, \veps_3)$ iff $\lm \ne 0$ and the term $\al_3$ in the sum $\al_3 + \lm^{-1} \al_{23} + \lm^{-2} \al_2$ is $\nu$-dominant, which means that $\lm^{-1} \al_{23} \leq \al_3$, $\lm^{-2} \al_2 \leq \al_3$, equivalently,
\[ \frac{\al^2_{23}}{\al_3^2} \leq \lm^2, \quad  \frac{\al_2}{\al_3} \leq \lm^2. \]
We conclude from \eqref{eq:6.14} that $\frac{\al_{23}^2}{\al_3^2} \geq \frac{\al_2}{\al_3}$,  and thus
\begin{equation}\label{eq:6.16} f_2 (\lm) = \CS (\veps_1, \veps_3) \dss \Iff \lm \geq \frac{\al_{23}}{\al_3}. \end{equation}
$\{$Recall that we initially assume  that $\CS (\veps_1, \veps_3) \ne 0$.$\}$

 We have seen that the intervals $[ 0, \frac{\al_2}{\al_{23}} ]$ and $[ \frac{\al_{23}}{\al_3}, \infty ]$ are the sets where the terms $\al_1$ and $\lm^2 \al_2$ in the sum \eqref{eq:6.9} are $\nu$-dominant and conclude that $[ \frac{\al_2}{\al_{23}}, \frac{\al_{23}}{\al_3} ]$ is the interval in which  the middle term $\lm \al_{23}$ is $\nu$-dominant. Note that in the border case $\al_2 \al_3 = \al_{23}^2$ this interval retracts to the single  point $\frac{\al_2}{\al_{23}} = \frac{\al_{23}}{\al_3}$.

We infer from \eqref{eq:6.7} and \eqref{eq:6.11} that in the interval $[0, \frac{\al_2}{\al_{23}}]$
\begin{equation}\label{eq:6.17} f_2 (\lm) = \CS (\veps_1, \veps_3) \frac{\lm^2 \al_3}{\al_2},  \end{equation}
and that in $[ \frac{\al_2}{\al_{23}}, \frac{\al_{23}}{\al_3} ]$
\begin{equation}\label{eq:6.18} f_2 (\lm) = \CS (\veps_1, \veps_3) \frac{\lm \al_3}{\al_{23}}. \end{equation}
Thus $f_2$  strictly increases on $[0, \frac{\al_2}{\al_{23}} ]$ from zero to
\[ f_2 \left(\frac{\al_2}{\al_{23}}\right) = \CS (\veps_1, \veps_3) \frac{\al_3}{\al_2} \left(\frac{\al_2}{\al_{23}}\right)^2 = \frac{\CS (\veps_1, \veps_3)}{\CS (\veps_2, \veps_3)}, \]
and then  strictly increases on $[\frac{\al_2}{\al_{23}}, \frac{\al_{23}}{\al_3} ]$ from this value to $\CS (\veps_1, \veps_3)$. Furthermore, we  infer from \eqref{eq:6.6} and \eqref{eq:6.10} that in the interval $[\frac{\al_2}{\al_{23}}, \frac{\al_{23}}{\al_3} ]$
\begin{equation}\label{eq:6.19} f_1 (\lm) = \CS (\veps_1, \veps_2) \frac{\al_2}{\lm \al_{23}}, \end{equation}
while in $[\frac{\al_{23}}{\al_3}, \infty ]$
\begin{equation}\label{eq:6.19.b} f_1 (\lm) = \CS (\veps_1, \veps_2) \frac{\al_2}{\lm^2\al_3}. \end{equation}
Thus $f_1$ strictly decreases on $[\frac{\al_2}{\al_{23}}, \frac{\al_{23}}{\al_3} ]$  from the value $\CS (\veps_1, \veps_2)$ to
\[ f_1 \left(\frac{\al_{23}}{\al_3}\right) = \CS (\veps_1, \veps_2) \frac{\al_2}{\al_3} \left(\frac{\al_3}{\al_{23}}\right)^2 = \frac{\CS (\veps_1, \veps_2)}{\CS ( \veps_2, \veps_3)} \]
and then on $[\frac{\al_{23}}{\al_3}, \infty ]$ again strictly decreases   from this value to zero. Note that the arguments $\lm = \frac{\al_2}{\al_{23}}, \frac{\al_{23}}{\al_3}$ correspond to the rays
\begin{equation}\label{eq:6.20} X_{23} : = \ray (\al_{23} \veps_2 + \al_2 \veps_3), \qquad  X_{32} : = \ray (\al_3 \veps_2 + \al_{23} \veps_3), \end{equation}
which in the case $\al_2 \al_3 <_{\nu} \al_{23}^2$ are the \textbf{critical rays} of $[ X_2, X_3 ]$ (cf. \cite{QF2}).

Summarizing the above study we obtain:

\begin{prop}\label{prop:6.1} Assume that $0 \leq \CS (\veps_1, \veps_2) \leq \CS (\veps_1, \veps_3)$ and that $\al_2 \al_3 \leq_{\nu} \al^2_{23}$.

\begin{itemize}
\item[a)] The function $f_1$ is constant on $[0, \frac{\al_2}{\al_{23}} ]$ with value $\CS (\veps_1, \veps_2)$ and strictly decreases  on $[ \frac{\al_2}{\al_{23}}, \infty ]$ from $\CS (\veps_1, \veps_2)$ to zero, with the intermediate value
\[ f_1 \left( \frac{\al_2}{\al_{23}} \right) = \frac{\CS (\veps_1, \veps_2)}{\CS (\veps_2, \veps_3)}, \]
provided that $\CS (\veps_1, \veps_2) \ne 0$. Otherwise $f_1 = 0$, whence $f_2 = f$ on $[0, \infty ]$.

\item[b)] The function $f_2$ strictly increases on $[0, \frac{\al_{23}}{\al_5} ]$  from zero to $\CS (\veps_1, \veps_3)$ and remains constant on $[\frac{\al_{23}}{\al_3}, \infty ]$, with intermediate value
    \[ f_2 \left( \frac{\al_2}{\al_{23}} \right) = \frac{\CS (\veps_1, \veps_3)}{\CS (\veps_2, \veps_3)}, \]
provided that $\CS (\veps_1, \veps_3) \ne 0$. Otherwise $f_1 = f_2 = f = 0$ on $[ 0, \infty ]$. 
\end{itemize}
\end{prop}

Since $\xi$ is the unique argument $\lm \in [0, \infty]$ with $f_1 (\lm) = f_2 (\lm)$, it follows from  Proposition ~\ref{prop:6.1} that $f_1 \geq f_2$ on $[0, \xi]$ and $f_1 \leq f_2$ on $[\xi, \infty]$, whence
\begin{equation}\label{eq:6.22} f = f_1 \; \mbox{on} \; [0, \xi] \; \dss{\mbox{and}} \; f = f_2 \; \mbox{on} \; [\xi, \infty]. \end{equation}
As seen below, the monotonicity behavior of $f$ is  determined by the location of $\xi$ with respect to the interval $[ \frac{\al_2}{\al_{23}}, \frac{\al_{23}}{\al_3}]$.

Since $f_2 = \CS (\veps_1, \veps_3) \geq$ $f_1$ on $[\frac{\al_{23}}{\al_3}, \infty]$ (cf. \eqref{eq:6.8}), it is clear that always
\begin{equation}\label{eq:6.23} \xi \leq \frac{\al_{23}}{\al_3}. \end{equation}
We have $\xi = 0$ iff $\CS (\veps_1, \veps_2) = 0$, and then $f = f_2$ strictly increases on $[0, \frac{\al_{23}}{\al_3}]$  to $\CS (\veps_1, \veps_3)$ and remains constant on $[ \frac{\al_{23}}{\al_3}, \infty ]$.

Assuming  that $\CS (\veps_1, \veps_2) > 0$, if $\xi \leq \frac{\al_2}{\al_{23}}$, then $f$ has the  constant value $\CS (\veps_1, \veps_2)$ on $[0, \xi]$. Thus, as follows from Proposition~\ref{prop:6.1}, $f$ strictly increases on $[\xi, \frac{\al_2}{\al_{23}} ]$ from $\CS (\veps_1, \veps_2)$ to $\frac{\CS (\veps_1, \veps_3)}{\CS (\veps_1, \veps_2)}$, and it strictly increases on $[ \frac{\al_2}{\al_{23}}, \frac{\al_{23}}{\al_3} ]$ from this value to $\CS (\veps_1, \veps_3)$.  Finally, $f$ remains constant on $[\frac{\al_{23}}{\al_3}, \infty ]$.

The graph of the function $f$ is illustrated as follows.
\begin{figure}[h]\label{fig:1}
\resizebox{0.5\textwidth}{!}{
\begin{tikzpicture}
\draw[thick] (0,0) node[below]{0} --(14,0)node[below]{$\lm$} ;
\draw[thick] (0,0) -- (0,11)node[left]{} node[above]{f};
\draw[blue, thick] (0,8) node[left]{$CS(\veps_1,\veps_3)$}--(13,8);
\draw[blue, thick] (0,5) node[left]{$\frac{CS(\veps_1,\veps_3)}{CS(\veps_2,\veps_3)}$}--(13,5);
\draw[blue, thick] (0,2) node[left]{${CS(\veps_1,\veps_2)}$}--(13,2);

\draw[blue, thick](10,0)node[below]{$\frac{\al_{23}}{\al_{2}}$}--(10,10);

\draw[blue, thick](2,0)node[below]{$\xi$}--(2,10);
\draw[blue, thick](3,0)node[below]{$\frac{\al_{2}}{\al_{23}}$}--(3,10);
\draw[help lines] (0,0) grid (13,10);
\draw[red, line width=0.5mm] (0,0) to[out=0, in=-120] (2,2) to [out=70, in=-100]  (3,5);
\draw[red, line width=0.5mm] (3,5) node[left]{}--(10,8);
\draw[red, line width=0.5mm] (10,8) --(13,8) node[right]{$\bf f_2$};

\draw[red, line width=0.5mm] (0,2) node[left]{}--(3,2);
\draw[red, line width=0.5mm] (3,2) to[out=-30, in=180]  (13,0.1) node[right]{$\bf f_1$};

\end{tikzpicture}}
\caption{}
\end{figure}

 We read off from this analysis that
\begin{equation}\label{eq:6.24} \xi < \frac{\al_2}{\al_{23}} \dss \Iff \CS (\veps_1, \veps_2) < \frac{\CS (\veps_1, \veps_2)}{\CS (\veps_2, \veps_3)}, \end{equation}
\begin{equation}\label{eq:6.25}
 \xi = \frac{\al_2}{\al_{23}} \dss \Iff \CS (\veps_1, \veps_2) = \frac{\CS (\veps_1, \veps_3)}{\CS (\veps_2, \veps_3)} .\end{equation}
In the remaining case that
\[\frac{\al_2}{\al_{23}} < \xi \leq \frac{\al_{23}}{\al_{2}} \]
we conclude by  Proposition \ref{prop:6.1} and \eqref{eq:6.22} that $f$ has the constant value $\CS (\veps_1, \veps_2)$ on $[0, \frac{\al_2}{\al_{23}} ]$, then strictly decreases  on $[\frac{\al_2}{\al_{23}}, \xi ]$ to a value $\rho : = f_1 (\xi) = f_2 (\xi)$ which we compute below. Then $f$ strictly increases  on $[\xi, \frac{\al_{23}}{\al_3} ]$ from the value $\rho$ to $\CS (\veps_1, \veps_3)$, and finally remains constant on $[\frac{\al_{23}}{\al_{3}}, \infty ]$. This implies that
\begin{equation}\label{eq:6.23p} \xi < \frac{\al_{23}}{\al_3}, \tag{\ref{eq:6.23}'}\end{equation}
improving \eqref{eq:6.23}.
By  \eqref{eq:6.12} and \eqref{eq:6.18} we have $\rho = f_2 (\xi)  = \CS (\veps_1, \veps_3) \frac{\al_{12}}{\al_{13}} \frac{\al_{2}}{\al_{23}}$, yielding
\begin{equation}\label{eq:6.25} \rho = \frac{\al_{12} \al_{13}}{\al_1 \al_{23}},  \end{equation}
whose square gives
\[ \rho^2 = \frac{\al_{12}^2 \al_{13}^2}{\al_1^2 \al_{23}^2} = \frac{\al_{12}^2}{\al_1 \al_{2}} \frac{\al_{13}^2}{\al_1 \al_{3}} \frac{\al_{2} \al_3}{\al_{23}^2}, \]
i.e.,
\begin{equation}\label{eq:6.26} \rho^2 = \frac{\CS (\veps_1, \veps_2) \CS (\veps_1, \veps_3)}{\CS (\veps_2, \veps_3)}. \end{equation}
It is now clear that $\xi > \frac{\al_2}{\al_{23}}$ iff $f_2 (\xi) < \CS (\veps_1, \veps_2)$ iff
\[ \frac{\CS (\veps_1, \veps_2) \CS (\veps_1, \veps_3)}{\CS (\veps_2, \veps_3)} < \CS (\veps_1, \veps_2)^2, \]
and thus
\begin{equation}\label{eq:6.27} \xi > \frac{\al_2}{\al_{23}} \dss\Iff \frac{\CS (\veps_1, \veps_3)}{\CS (\veps_2, \veps_3)} < \CS (\veps_1, \veps_2). \end{equation}

Recall that we have assumed  (cf. \eqref{eq:6.5}) that $0 < \CS (\veps_1, \veps_2) \leq \CS (\veps_1, \veps_3)$. Assuming further that
\begin{equation}\label{eq:6.28.a} \frac{\CS (\veps_1, \veps_3)}{\CS (\veps_2, \veps_3)} < \CS (\veps_1, \veps_2), \end{equation}
we still need to distinguish the cases $\CS (\veps_1, \veps_2) < \CS (\veps_1, \veps_3)$ and $\CS (\veps_1, \veps_2) = \CS (\veps_1, \veps_3)$ (where \eqref{eq:6.28.a} holds automatically).

This means that either $f(0) < f (\infty)$ or $f(0) = f (\infty)$, which we judge as a difference in the monotonicity behavior of $f$. The graph of $f$ is illustrated  in Figure 2.

\begin{figure}[h!]\label{fig:2}
\resizebox{1\textwidth}{!}{

\begin{tikzpicture}[]
\draw[thick] (0,0) node[below]{0} --(14,0)node[below]{$\lm$} ;
\draw[thick] (0,0) -- (0,11)node[left]{} node[above]{f};
\draw[blue, thick] (0,8) node[left]{$CS(\veps_1,\veps_3)$}--(13,8);
\draw[blue, thick] (0,3) node[left]{$\frac{CS(\veps_1,\veps_3)}{CS(\veps_2,\veps_3)}$}--(13,3);
\draw[blue, thick] (0,4.9) node[left]{$\rho$}--(13,4.9);
\draw[blue, thick] (0,6) node[left]{${CS(\veps_1,\veps_2)}$}--(13,6);

\draw[blue, thick](9,0)node[below]{$\frac{\al_{23}}{\al_{2}}$}--(9,10);

\draw[blue, thick](4.6,0)node[below]{$\xi$}--(4.6,10);
\draw[blue, thick](2,0)node[below]{$\frac{\al_{2}}{\al_{23}}$}--(2,10);
\draw[help lines] (0,0) grid (13,10);
\draw[red, line width=0.5mm] (0,0) to[out=0, in=-100] (2,3) ;
\draw[red, line width=0.5mm] (2,3) node[left]{}--(9,8);
\draw[red, line width=0.5mm] (9,8) --(13,8) node[right]{$\bf f_2$};

\draw[red, line width=0.5mm] (0,6) node[left]{}--(2,6);
\draw[red, line width=0.5mm] (2,6) node[left]{}--(9,3);
\draw[red, line width=0.5mm] (9,3) to[out=-60, in=180]  (13,0.1) node[right]{$\bf f_1$};

\end{tikzpicture}

\begin{tikzpicture}[]
\draw[thick] (0,0) node[below]{0} --(14,0)node[below]{$\lm$} ;
\draw[thick] (0,0) -- (0,11)node[left]{} node[above]{f};
\draw[blue, thick] (0,8) node[left]{$CS(\veps_1,\veps_2)$}--(13,8);
\draw[blue, thick] (0,3) node[left]{$\frac{CS(\veps_1,\veps_3)}{CS(\veps_2,\veps_3)}$}--(13,3);
\draw[blue, thick] (0,5.5) node[left]{$\rho$}--(13,5.5);

\draw[blue, thick](9,0)node[below]{$\frac{\al_{23}}{\al_{2}}$}--(9,10);

\draw[blue, thick](5.5,0)node[below]{$\xi$}--(5.5,10);
\draw[blue, thick](2,0)node[below]{$\frac{\al_{2}}{\al_{23}}$}--(2,10);
\draw[help lines] (0,0) grid (13,10);
\draw[red, line width=0.5mm] (0,0) to[out=0, in=-100] (2,3) ;
\draw[red, line width=0.5mm] (2,3) node[left]{}--(9,8);
\draw[red, line width=0.5mm] (9,8) --(13,8) node[right]{$\bf f_2$};

\draw[red, line width=0.5mm] (0,8) node[left]{}--(2,8);
\draw[red, line width=0.5mm] (2,8) node[left]{}--(9,3);
\draw[red, line width=0.5mm] (9,3) to[out=-60, in=180]  (13,0.1) node[right]{$\bf f_1$};

\end{tikzpicture}
}
\caption{A. $(\CS (\veps_1, \veps_2) < \CS (\veps_1, \veps_3))$, \qquad B. $(\CS (\veps_1, \veps_2) = \CS (\veps_1, \veps_3)).$}
\end{figure}

 We summarize the above  analysis of $f(\lm) = \CS (\veps_1, \veps_2 + \lm \veps_2)$ for $\CS (\veps_2, \veps_3) > e$ as follows.

\begin{thm}\label{thm:6.2} Assume that $0 \leq \CS (\veps_1, \veps_2) \leq \CS (\veps_1, \veps_3)$ and that $\al_2 \al_3 <_{\nu} \al^2_{23}$.

\begin{itemize}
\item[a)] If $\CS (\veps_1, \veps_3) = 0$, then $f = 0$. If $\CS (\veps_1, \veps_2) = 0$ and $\CS (\veps_1, \veps_3) > 0$, then $f = 0$ on $[0, \frac{\al_2}{\al_{23}} ]$, but $f$ strictly increases on $[ \frac{\al_2}{\al_{23}}, \frac{\al_{23}}{\al_3} ]$  from zero to $\CS (\veps_1, \veps_3)$ and finally remains constant on $[\frac{\al_{23}}{\al_3}, \infty ]$.
\item[b)] Assume now that
\[ 0 < \CS (\veps_1, \veps_2) \leq \CS (\veps_1, \veps_3). \]
In the case\\
 \begin{equation}
   \CS (\veps_1, \veps_2) < \frac{\CS (\veps_1, \veps_3)}{\CS (\veps_2, \veps_3)} < \CS (\veps_1, \veps_3)\tag{A}
 \end{equation}
the function $f$ increases on $[0, \infty ]$ monotonically from $\CS (\veps_1, \veps_2)$ to $\CS (\veps_1, \veps_3)$. More precisely, $\frac{\al_{12}}{\al_{13}} < \frac{\al_2}{\al_{23}}$, and $f$ has  constant   value $\CS (\veps_1, \veps_2)$ on $[0, \frac{\al_{12}}{\al_{13}} ]$, then it  strictly increases  on $[\xi, \frac{\al_2}{\al_{23}} ]$ to the value $\frac{\CS (\veps_1, \veps_3)}{\CS (\veps_2, \veps_3)}$ and strictly increases again on $[ \frac{\al_2}{\al_{23}}, \frac{\al_{23}}{\al_3} ]$ to the value $\CS (\veps_1, \veps_3)$ for which it remains constant on $[ \frac{\al_{23}}{\al_3}, \infty ]$.

In the border case
 \begin{equation}
  \CS (\veps_1, \veps_2) = \frac{\CS (\veps_1, \veps_3)}{\CS (\veps_2, \veps_3)}, \tag{$\partial A$}\end{equation}
we have $\frac{\al_{12}}{\al_{13}} = \frac{\al_2}{\al_{23}}$, where $f$ is constant of value $\CS (\veps_1, \veps_2)$ on $[0, \frac{\al_2}{\al_{23}} ]$, strictly increases  on $[ \frac{\al_2}{\al_{23}}, \frac{\al_{23}}{\al_3} ]$ from $\CS (\veps_1, \veps_2)$ to $\CS (\veps_1, \veps_3)$, and finally is constant of value $\CS (\veps_1, \veps_3)$ on $[ \frac{\al_{23}}{\al_3}, \infty ]$.

In the remaining case
 \begin{equation}
\frac{\CS (\veps_1, \veps_3)}{\CS (\veps_2, \veps_3)} < \CS (\veps_1, \veps_2) < \CS (\veps_1, \veps_3) \tag{B} \end{equation}
and its border case  \begin{equation}
 0 < \CS (\veps_1, \veps_2) = \CS (\veps_1, \veps_3) \tag{$\partial B$}\end{equation}
the function $f$ attains its minimal value $\rho$ at the  unique point $\lm = \xi = \frac{\al_{12}}{\al_{13}}$, where
\[ \frac{\al_2}{\al_{23}} < \xi < \frac{\al_{23}}{\al_3}. \]
Explicitly, $f$ has the constant value $\CS (\veps_1, \veps_2)$ on $[0, \frac{\al_2}{\al_{23}} ]$, it strictly decreases  on $[ \frac{\al_2}{\al_{23}}, \xi ]$ to the value
\[  \rho : = \frac{\al_{12} \al_{13}}{\al_1 \al_{23}} = \sqrt{\frac{\CS (\veps_1, \veps_2) \CS (\veps_1, \veps_3)}{\CS (\veps_2, \veps_3)}}, \]
then it strictly increases  on $[ \xi, \frac{\al_{23}}{\al_3} ]$ to the value $\CS (\veps_1, \veps_3)$ and  remains constant on ~$[\frac{\al_{23}}{\al_3}, \infty ]$.
\end{itemize}
\end{thm}
 Finally we discuss the behavior of $f(\lm)$ in the easier case that $\al_{23}^2 \leq_{\nu} \al_2 \al_3$, assuming  as before that $0 \leq \CS (\veps_1, \veps_2) \leq \CS (\veps_1, \veps_3)$. Formula \eqref{eq:6.9} for $q (\veps_2 + \lm \veps_3)$ simplifies to
\[ q (\veps_2 + \lm \veps_3) = \al_2 + \lm^2 \al_3, \]
and so
\begin{equation}\label{eq:6.28} f_1 (\lm) = \frac{\al_{12}^2}{\al_1 (\al_2 + \lm^2 \al_3)}, \qquad f_2 (\lm) = \frac{\lm^2 \al_{13}^2}{\al_1 (\al_2 + \lm^2 \al_3)}. \end{equation}
If $\CS (\veps_1, \veps_3) = 0$, i.e., $\al_{13} = 0$, then $f_1 = 0$, $f_2 = 0$, $f = 0$. Henceforth we assume that $\CS (\veps_1, \veps_3) > 0$ (but allow that $\CS (\veps_1, \veps_2) = 0$). We read off from \eqref{eq:6.28} that, if $\lm^2 \leq \frac{\al_2}{\al_3}$, then
\begin{equation}\label{eq:6.29} f_1 (\lm)  =  \frac{\al_{12}^2}{\al_1 \al_2} = \CS (\veps_1, \veps_2), \qquad
\mbox{$f$}_2 (\lm)  =  \frac{\lm^2 \al_{13}^2}{\al_1 \al_2} = \lm^2 \frac{\al_3}{\al_2} \CS (\veps_1, \veps_3), \end{equation}
while, if $\lm^2 \geq \frac{\al_2}{\al_3}$, then
\begin{equation}\label{eq:6.30}
f_1 (\lm) = \frac{\al_{12}^2}{\lm \al_1 \al_3} = \lm^{-2} \frac{\al_2}{\al_3} \CS (\veps_1, \veps_2), \qquad  f_2 (\lm) = \frac{\al_{13}^2}{\al_1 \al_3} = \CS (\veps_1, \veps_3). \end{equation}
For the unique point $\lm = \frac{\al_{12}}{\al_{13}} = \xi$ where $f_1 (\lm) = f_2 (\lm)$ we have
\begin{equation}\label{eq:6.31} \xi^2 = \frac{\al_{12}^2}{\al_{13}^2} \leq \frac{\al_2}{\al_3}. \end{equation}

The case $\xi^2 = \frac{\al_2}{\al_3}$ means that $\CS (\veps_1, \veps_2) = \CS (\veps_1, \veps_3)$, for which  for $\lm^2 \leq \frac{\al_2}{\al_3}$ we have
\[ f_1 (\lm) = \CS (\veps_1, \veps_2) \geq \mbox{$f$}_2 (\lm), \]
and for $\lm^2 \geq \frac{\al_2}{\al_3}$ we have
\[ f_2 (\lm) = \CS (\veps_1, \veps_3) \geq \mbox{$f$}_1 (\lm). \]
Thus $f = \max (f_1, f_2)$ has constant value $\CS (\veps_1, \veps_2) = \CS (\veps_1, \veps_3)$ on $[0, \infty]$.

We are left with the case
\[ 0 \leq \CS (\veps_1, \veps_2) < \CS (\veps_1, \veps_3). \]
Let $\sqrt{\frac{\al_2}{\al_3}}$ denote the square root of $\frac{\al_2}{\al_3}$ in the ordered abelian group $\tG^{\frac{1}{2}} \supset \tG$.
We learn from ~\eqref{eq:6.29} that $f_1 (\lm) \geq f_2 (\lm)$ if $0 \leq \lm \leq \xi$, while $f_1 (\lm) \leq f_2 (\lm)$ if $\xi \leq \lm \leq \infty$,\footnote{Actually we know this for long, cf. the arguments following \eqref{eq:6.12}.} and thus obtain
\begin{equation}\label{eq:6.32} f (\lm) = \left\{ \begin{array}{ll}
\CS (\veps_1, \veps_2) & 0 \leq \lm \leq \xi, \\
\lm^2 \frac{\al_3}{\al_2} \CS (\veps_1, \veps_3) & \xi \leq \lm \leq \sqrt{ \frac{\al_2}{\al_3}}, \\
\CS (\veps_1, \veps_3) & \sqrt{ \frac{\al_2}{\al_3}} \leq \lm \leq \infty.  \end{array} \right. \end{equation}

We summarize all this as follows.

\begin{thm}\label{thm:6.3} Assume that $\al_{23}^2 \leq_{\nu} \al_2 \al_3$ and $\CS (\veps_1, \veps_2) \leq \CS (\veps_1, \veps_3)$.
\begin{itemize}
\item[a)] If $\CS (\veps_1, \veps_2) = \CS (\veps_1, \veps_3)$, then $f$ is constant on $[0, \infty]$ with value $\CS (\veps_1, \veps_2)$.
\item[b)] If $\CS (\veps_1, \veps_2) < \CS (\veps_1, \veps_3)$, then $\xi < \sqrt{\frac{\al_2}{\al_3}}$. If $\sqrt{\frac{\al_2}{\al_3}} \in \tG$, then $f$ is constant with value $\CS (\veps_1, \veps_2)$ on $[0, \xi]$ (in particular $\xi = 0$ if $\CS (\veps_1, \veps_2) = 0$), further increases strictly from $\CS (\veps_1, \veps_2)$ to $\CS (\veps_1, \veps_3)$ on $\big[\xi, \sqrt{\frac{\al_2}{\al_3}} \big]$, and then remains constant of value $\CS (\veps_1, \veps_3)$ on $\big[ \sqrt{\frac{\al_2}{\al_3}}, \infty \big]$.
     If $\sqrt{\frac{\al_2}{\al_3}} \not\in \tG$, this holds again after replacing these intervals by $\big[\xi, \sqrt{\frac{\al_2}{\al_3}} \big[ \; : = \big \{ \lm \in \tG \ds  \vert \xi \leq \lm < \sqrt{\frac{\al_2}{\al_3}} \big\}$ and \  $ \big ] \sqrt{\frac{\al_2}{\al_3}}, \infty \big ]:= \big\{ \lm \in \tG \ds \vert \sqrt{\frac{\al_2}{\al_3}} < \lm < \infty \big\} \cup \{ \infty \}$.
\end{itemize}
\end{thm}
In the case $\sqrt{\frac{\al_2}{\al_3}} \in \tG$ the graph of $f(\lm)$ with respect  to the variable $\lm^2$ looks as follows.

 \begin{figure}[h]\label{fig:3}
\resizebox{0.5\textwidth}{!}{

\begin{tikzpicture}[]
\draw[thick] (0,0) node[below]{0} --(14,0)node[below]{$\lm^2$} ;
\draw[thick] (0,0) -- (0,11)node[left]{} node[above]{f};
\draw[blue, thick] (0,8) node[left]{$CS(\veps_1,\veps_3)$}--(13,8);
\draw[blue, thick] (0,4) node[left]{${CS(\veps_1,\veps_2)}$}--(13,4);

\draw[blue, thick](9,0)node[below]{$\frac{\al_{2}}{\al_{3}}$}--(9,10);

\draw[blue, thick](4.5,0)node[below]{$\frac{\al_{12}^2}{\al_{13}^2}$}--(4.5,10);
\draw[help lines] (0,0) grid (13,10);
\draw[red, line width=0.5mm] (0,0) node[left]{}--(9,8);
\draw[red, line width=0.5mm] (9,8) --(13,8) node[right]{$\bf f_2$};

\draw[red, line width=0.5mm] (0,4) node[left]{}--(9,4);
\draw[red, line width=0.5mm] (9,4) to[out=-60, in=180]  (13,0.1) node[right]{$\bf f_1$};

\end{tikzpicture}
}
\caption{The case of  $\sqrt{\frac{\al_3}{\al_3}} \in \tG.$}
\end{figure}

\section{The CS-profiles on a ray interval}\label{sec:7}

As before we assume that $eR$ is a (bipotent) semifield. Given anisotropic rays $Y_1, Y_2, W$ in the $R$-module $V$ (i.e., $Y_1, Y_2, W \in \Ray (V_{\an})$, $V_{\an} : = \{ x \in V \ds \vert q (x) \ne 0 \} \cup \{ 0 \}$) with $Y_1 \ne Y_2$), we are interested in the $\CS$-\textit{profile of $W$ on the interval} $[Y_1, Y_2]$, by which we mean the monotonicity behavior of the function $\CS (W, \udscr)$ on $[Y_1, Y_2]$ with respect to the total ordering \footnote{This also includes information about the zero set of this function.} $\leq_{Y_1}$, as studied in \S\ref{sec:6}. (There we labeled $W = X_1$, $Y_1 = X_2$, $Y_2 = X_3$.)

More succinctly we denote the set $[Y_1, Y_2]$, equipped with the total ordering $\leq_{Y_1}$, by $\overrightarrow{[Y_1, Y_2]}$, and call it an \textbf{oriented closed ray interval}. Often we use the shorter term ``$W$-\textbf{profile on} $\overrightarrow{[Y_1, Y_2]}$'' instead of ``$\CS$-profile of $W$ on $\overrightarrow{[Y_1, Y_2]}$'', whenever it is clear from the context that we are dealing with $\CS$-ratios.

\begin{defn}\label{def:7.1} Let $W, Y_1, Y_2 \in \Ray (V_{\an})$ be anisotropic, and assume that $Y_1 \ne Y_2$.

\begin{itemize}\dispace
\item[a)] We call a $W$-profile on $\overrightarrow{[Y_1, Y_2]}$ \textbf{ascending}, if $\CS (W, Y_1) \leq \CS (W, Y_2)$, and \textbf{descending} if $\CS (W, Y_1) \geq \CS (W, Y_2)$.
\item[b)] We say that a $W$-profile on $\overrightarrow{[Y_1, Y_2]}$ is \textbf{monotone}, if it is either increasing\footnote{In basic terms, it  means that $\CS (W, Z_1) \leq \CS (W, Z_2)$ for all $Z_1, Z_2 \in [Y_1, Y_2]$ with $[Y_1, Z_1] \subset [Y_1, Z_2]$.} or decreasing, and then usually speak of a ``monotone $W$-profile on $[Y_1, Y_2]$'', omitting the arrow indicating orientation, since it is irrelevant.
\item[c)] We say that a $W$-profile on $\overrightarrow{[Y_1, Y_2]}$ is \textbf{positive}, if $\CS (W, Y_1) > 0$, $\CS (W, Y_2) > 0$, and so $\CS (W, Z) > 0$ for all $Z \in [Y_1, Y_2]$, and we say that the $W$-profile is \textbf{non-positive} (or ``attains zero''), if $\CS (W, Y_1) = 0$ or $\CS (W, Y_2) = 0$.
\end{itemize}
\end{defn}

\begin{rem}\label{rem:7.2} It is clear from \S\ref{sec:6} that if, say, $\CS (W, Y_1) = 0$ and  $\CS (W, Y_2) > 0$,  then $\CS (W, Z) > 0$ for all $Z \ne Y_1$ in $[Y_1, Y_2]$, while, if $\CS (W, Y_1) = \CS (W, Y_2) = 0$,  then $\CS (W, Z) = 0$ for all $Z \in [Y_1, Y_2]$. \end{rem}

 We learned in \S\ref{sec:6} (cf. Theorems \ref{thm:6.2} and \ref{thm:6.3}) that the monotonicity behavior of the function $\CS (W, \udscr)$ on $\overrightarrow{[Y_1, Y_2]}$ is essentially determined\footnote{In the case $\CS (Y_1, Y_2) \leq e$ we also need the information whether the square class $eq (Y_1) q(Y_2) \subset \tG$ of $eR$ is trivial or not (cf. Theorem \ref{thm:6.3}).}  by the three ratios $\CS (Y_1, Y_2)$, $\CS (W, Y_1)$, $\CS (W, Y_2)$, more precisely by certain strict inequalities ($<$) and equalities (=) involving these $\CS$-ratios. Accordingly, we  classify the $\CS$-profiles on $\overrightarrow{[Y_1, Y_2]}$ by characterizing  them into ``basic types'', each is given by a conjunction $T$ of inequalities involving these $\CS$-ratios and zero.
From Theorems \ref{thm:6.2} and \ref{thm:6.3} we gain the following list of ``\textbf{basic ascending  types}'', for which every ascending $W$-profile on $\overrightarrow{[Y_1, Y_2]}$ belongs to exactly one type $T$, and the condition~ $T$ encodes completely the monotonicity behavior of the function $\CS (W,\udscr)$ on $[Y_1, Y_2]$.\footnote{The reader may argue that our notion of basic type lacks a precise definition. We can remedy this by \textit{defining} the basic types on $[Y_1, Y_2]$ as all the conditions $A$, $\partial A$, $B$, $\dots$ appearing in Tables \ref{table:7.3}, \ref{table:7.4} and Scholium \ref{schol:7.5} below.}

\begin{tabl}\label{table:7.3} Assume that $\CS (Y_1, Y_2) > e$.
\begin{itemize}
\item[a)] The positive ascending basic types (i.e., types of positive ascending $W$-profiles) on $\overrightarrow{[Y_1, Y_2]}$ are
\begin{itemize}\dispace
  \item[$\phantom{\partial} A$:] \quad  $0 < \CS (W, Y_1) < \frac{\CS (W, Y_2)}{\CS (Y_1, Y_2)}$,
\item[$\partial A$:] \quad  $0 < \CS (W, Y_1) = \frac{\CS (W, Y_2)}{\CS (Y_1, Y_2)}$,
\item[$\phantom{\partial} B$:] \quad $0 < \frac{\CS (W, Y_2)}{\CS (Y_1, Y_2)} < \CS (W, Y_1) < \CS (W, Y_2)$,
    \item[$\partial B$:] \quad $0 < \CS (W, Y_1) = \CS (W, Y_2)$.
\end{itemize}
(Note that $\partial B$ implies $0 < \frac{\CS (W, Y_2)}{\CS (Y_1, Y_2)} < \CS (W, Y_1)$. So we could also write \\
$\partial B$: \quad $0 < \frac{\CS (W, Y_2)}{\CS (Y_1, Y_2)} < \CS (W, Y_1) = \CS (W, Y_2)$.)
\item[b)] The non-positive ascending basic types are
\begin{itemize}
\item[$\phantom{\partial} E$:] \quad $0 = \CS (W, Y_1) < \CS (W, Y_2)$,
\item[$\partial E$:] \quad $0 = \CS (W, Y_1) = \CS (W, Y_2)$.
\end{itemize}
\end{itemize}
All these types are increasing, i.e., determine increasing $W$-profiles, except $B$ and $\partial B$.
\end{tabl}

\begin{tabl}\label{table:7.4} Assume that $\CS (Y_1, Y_2) \leq e$ and $Y_1 \ne Y_2$. We have the following list of ascending basic types on $\overrightarrow{[Y_1, Y_2]}$.
\begin{itemize}
\item[a)] The positive ascending basic types\\
$\phantom{\partial} C$: \quad $0 < \CS (W, Y_1) < \CS (W, Y_2)$, \\
$\partial C$: \quad $0 < \CS (W, Y_1) = \CS (W, Y_2)$.
\item[b)] The non-positive ascending basic types\footnote{Although $D$ and $\partial D$ are the same sentences as $E$ and $\partial E$ in Table~\ref{table:7.3}, we use a different letter ``$D$'', since we include in the type the information whether $\CS (Y_1, Y_2) >e$ or $\CS (Y_1, Y_2) \leq e$.}\\
$\phantom{\partial} D$: \quad $0 = \CS (W, Y_1) < \CS (W, Y_2)$,\\
$\partial D$: \quad $0 = \CS (W, Y_1) = \CS (W, Y_2)$.
\end{itemize}
Note that all these types are increasing.
\end{tabl}

For each basic type $T$ there is a \textbf{reverse type} $T'$, obtained by interchanging $Y_1$ and $Y_2$ in condition $T$. Thus the reverses of the types listed in Tables~\ref{table:7.3} and \ref{table:7.4} exhaust all basic types of descending $\CS$-profiles on $\overrightarrow{[Y_1, Y_2]}$. We obtain the following list of such types.

\begin{schol}\label{schol:7.5} $ $ \begin{itemize}

   \item[a)] When $\CS (Y_1, Y_2) > e$, the descending basic types are
$A'$, $\partial A' : = (\partial A)'$, $B'$, $\partial B' : = (\partial B)' = \partial B$. $E'$, $\partial E' : = (\partial E)' = \partial E$.

\item[b)] When $\CS (Y_1, Y_2) \leq e$, the descending basic types are
$C'$, $\partial C' : = \partial C$, $D'$, $\partial D' : = (\partial D)' = \partial D$.
\end{itemize}
All these types are decreasing except $B'$ and $\partial B'$, which are not monotone. \end{schol}

 In later sections additional conditions on $\CS (Y_1, Y_2)$, $\CS (W, Y_1)$, $\CS (W, Y_2)$ will come into play, which  arise from a basic type $T$ by relaxing the strict inequality sign $<$ to $\leq$ at one or several places. We name such  condition $U$  a \textbf{relaxation} of $T$, and  call all the arising  relaxations the \textbf{composed} $\CS$-\textbf{types} on $[Y_1, Y_2]$. The reason for the latter term is that such a relaxation $U$ is a disjunction
\begin{equation}\label{eq:7.1} U = T_1 \vee \dots \vee T_r \end{equation}
of several basic types $T_i$, as  will be seen (actually with $r \leq 4$). Since every $W$-profile on $\overrightarrow{[Y_1, Y_2]}$ belongs to exactly one basic type $T$, it is then obvious that the $T_i$ in \eqref{eq:7.1} are uniquely determined by $U$ up to permutation, $T$ being one of them.
We call the $T_i$ the \textbf{components} of the relaxation $U$.

We extend part of the terminology of basic types to their relaxations in the obvious way. The \textbf{reverse type} $U'$ of $U$ arises by interchanging $Y_1$ and $Y_2$ in the condition $U$. The composed type $U$ is \textbf{ascending} (resp. \textbf{descending}), if the sentence $\CS (W, Y_1) \leq \CS (W, Y_2)$ (resp. $\CS (W, Y_1) \geq \CS (W, Y_2)$) is a consequence of $U$, and $U$ is \textbf{positive}, if $U$ implies $0 < \CS (W, Y_i)$ for $i = 1,2$.

We  list out all relaxations of all ascending basic types on $\overrightarrow{[Y_1, Y_2]}$, at first in the case $\CS (Y_1, Y_2)> e$, an then in the case $\CS (Y_1, Y_2) \leq e$. It will turn out that all these relaxations are again ascending.

\begin{schol}\label{schol:7.6} Assume that $\CS (Y_1, Y_2) > e$.
\begin{itemize}
\item[a)] The basic type $E : 0 = \CS (W, Y_1) < \CS (W, Y_2)$ has only one relaxation
\[ \overline{E} : \quad 0 = \CS (W, Y_1) \leq \CS (W, Y_2), \]
for which  $\overline{E} = E \vee \partial E$.
\item[b)] The basic type $A : 0 < \CS (W, Y_1) < \frac{\CS (W, Y_2)}{\CS (Y_1, Y_2)}$ has the relaxations
$$\begin{array}{rll}
A_0:  & 0 \leq \CS (W, Y_1) < \frac{\CS (W, Y_2)}{\CS (Y_1, Y_2)}, \\[2mm]
  \overline{A}:   & 0 < \CS (W, Y_1) \leq \frac{\CS (W, Y_2)}{\CS (Y_1, Y_2)},\\[2mm]
\overline{A}_0:  & 0 \leq \CS (W, Y_1) \leq \frac{\CS (W, Y_2)}{\CS (Y_1, Y_2)}.
\end{array}$$
For these relaxations we have
$$\begin{array}{lll}
A_0 &= A \vee (0 = \CS (W, Y_1) < \CS (W, Y_2)) = A \vee E\\[2mm]
\overline{A} &= A \vee (0 = \CS (W, Y_1) < \CS (W, Y_2)) = A \vee \partial A\\[2mm]
\overline{A}_0 &= \overline{A} \vee (0 = \CS (W, Y_1) \leq \frac{\CS (W, Y_2)}{\CS ( Y_1, Y_2)}) = \overline{A} \vee \overline{E} = A \vee \partial A \vee E \vee \partial E.
\end{array}$$
\item[c)] The positive relaxations of $B : 0 < \frac{\CS (W, Y_2)}{\CS (Y_1, Y_2)} < \CS (W, Y_1) < \CS (W, Y_2)$
are
$$\begin{array}{rll}
  \overline{B}: & 0 < \frac{\CS (W, Y_2)}{\CS (Y_1, Y_2)} \leq \CS (W, Y_1) < \CS (W, Y_2),\\
\tlB:& 0 < \frac{\CS (W, Y_2)}{\CS (Y_1, Y_2)} < \CS (W, Y_1) \leq \CS (W, Y_2),\\
\widetilde{\overline{B}}: & 0 < \frac{\CS (W, Y_2)}{\CS (Y_1, Y_2)} \leq \CS (W, Y_1) \leq \CS (W, Y_2).
\end{array}$$

We have $\widetilde{\overline{B}} = \overline{B} \vee \tlB$, since $\CS (W, Y_1) = \CS (W, Y_2)$ implies $\frac{\CS (W, Y_2)}{\CS (Y_1, Y_2)} < \CS (W, Y_1)$, and $\frac{\CS (W, Y_2)}{\CS (Y_1, Y_2)} = \CS (W, Y_1)$ implies $\CS (W, Y_1) < \CS (W, Y_2)$. Also
\begin{equation}\label{eq:7.2} \tlB = B \vee (0 < \CS (W, Y_1) = \CS (W, Y_2)) = B \vee \partial B, \end{equation}
since $0 < \CS (W, Y_1) = \CS (W, Y_2)$ implies $0 < \frac{\CS (W, Y_2)}{\CS (W, Y_2)} < \CS (W, Y_1)$,\\
and
\begin{equation}\label{eq:7.3} \overline{B} = B \vee (0 < \CS (W, Y_1) = \frac{\CS (W, Y_2)}{\CS (W, Y_1)}) = B \vee \partial A. \end{equation}
Finally
\begin{equation}\label{eq:7.4} \widetilde{\overline{B}} = \overline{B} \vee \tlB = B \vee \partial B \vee \partial A . \end{equation}

We obtain the non-positive relaxations of $B$ by replacing in all these sentences $B, \overline{B}, \tlB, \widetilde{\overline{B}}$ the part $\big(0 < \frac{\CS (W, Y_2)}{\CS (Y_1, Y_2)}\big)$ by
$\big(0 \leq \frac{\CS (W, Y_2)}{\CS (Y_1, Y_2)}\big) = \big(0 < \frac{\CS (W, Y_2)}{\CS (Y_1, Y_2)}\big) \vee \big (0 = \CS (W, Y_2)\big)$.
Since all 4 sentences $B, \overline{B}, \tlB, \widetilde{\overline{B}}$ have the consequence $\CS (W, Y_1) \leq \CS (W, Y_2)$, we get 4 non-positive relaxations of $B$, namely
\begin{equation}\label{eq:7.5}
\begin{array}{lll}
B_0 & = B \vee (0 = \CS (W, Y_1) = \CS (W, Y_2)) = B \vee \partial E,\\[1mm]
\overline{B}_0 & = \overline{B} \cup \partial E, \\[1mm]
\tlB_0 & = \tlB \vee \partial E, \\
\widetilde{\overline{B}}_0 &= \tlB \cup \partial E = B \vee \partial B \vee \partial A \vee \partial E.\end{array}
 \end{equation}
\end{itemize}
\end{schol}
\begin{rem}\label{rem:7.7} Assume that $\CS (Y_1, Y_2) > e$.
\begin{itemize}
\item[a)] We have $$
\begin{array}{lll} A_0 \wedge \overline{A} &= (A \vee E) \wedge (A \vee \partial E) = A, \\[2mm]
\overline{B} \wedge \tlB & = (B \vee \partial A) \wedge (B \vee \partial B) = B,\end{array}$$
since different basic types are incompatible (= contradictory).
\item[b)] The condition
$$\Asc : = 0 \leq \CS (W, Y_1) \leq \CS (W, Y_2)$$
is the disjunction of all ascending basic types,
\begin{equation}\label{eq:7.9} \Asc = A \vee \partial A \vee B \vee \partial B \vee E \vee \partial E, \end{equation}
since every $W \in \Ray (V)$ fulfills exactly one of the basic type sentences.\\
The condition
$$\Asc^+ : = 0 < \CS (W, Y_1) < \CS (W, Y_2)$$
is the disjunction of all positive strictly ascending basic types,
\begin{equation}\label{eq:7.10} \Asc^+ = A \vee \partial A \vee B. \end{equation}
\end{itemize}
Note that both $\Asc$ and $\Asc^+$ are not relaxations of basic types, and thus are not regarded as composite types.
\end{rem}

 The table of composite types in the case $\CS (Y_1, Y_2) \leq e$ is much simpler.

\begin{schol}\label{schol:7.8} Assume that $\CS (Y_1, Y_2) \leq e$. The basic type $$D : 0 = \CS (W, Y_1) < \CS (W, Y_2)$$ has only one relaxation:
\[ \overline{D} = D \cup \partial D : 0 = \CS (W, Y_1) \leq \CS (W, Y_2). \]
The basic type $C : 0 < \CS (W, Y_1) < \CS (W, Y_2)$ has three relaxations, namely
\[ \begin{array}{lcl}
\overline{C} & : = & C \vee \partial C : 0 < \CS (W, Y_1) \leq \CS (W, Y_2),\\[2mm]
C_0 & : = & C \vee D : 0 \leq \CS (W, Y_1) < \CS (W, Y_2),\\[2mm]
\overline{C}_0 & := & \overline{C} \vee D = \overline{C} \vee \overline{D} : 0 \leq \CS (W, Y_1) \leq \CS (W, Y_2).
 \end{array} \]
$\overline{C}_0$ is the disjunction of all increasing basic types on $\overrightarrow{[Y_1, Y_2]}$. \end{schol}

\section{A convexity lemma for linear CS-inequalities, and first applications}\label{sec:8}

As before we assume that $eR$ is a semifield and $(q, b)$ is a quadratic pair on an $R$-module ~$V$ with $q$ anisotropic.

\begin{lem}\label{lem:8.1} Given $w_1, \dots, w_n \in V \setminus \{ 0 \}$ and $\lm_1, \dots, \lm_n \in \tG$,  let $w : = \sum\limits_{i = 1}^n \lm_i w_i$. For any $y \in V \setminus \{ 0 \}$ the following holds
\begin{equation}\label{eq:8.1} \CS (w, y) = \sum\limits_{i = 1}^n \CS (w_i, y) \al_i \end{equation}
with $\al_i \in \tG$, $0 < \al_i \leq e$, namely
\begin{equation}\label{eq:8.2} \al_i : = \frac{q (\lm_i w_i)}{q(w)}, \end{equation}
and thus  $0 < \al_i \leq e$.
\end{lem}

\begin{proof}
 Since $b (w, y) = \sum\limits_{i = 1}^n b (\lm_i w_i, y)$, we have
\[
\CS (w, y)  =  \frac{b (w, y)^2}{q(w) q(y)} = \sum\limits_{i = 1}^n \frac{b (\lm_i w_i, y)^2}{q(w) q(y)}
 =  \sum\limits_{i = 1}^n \frac{b (\lm_i w_i, y)^2}{q(\lm_i w_i) q(y)} \; \frac{q(\lm_i w_i)}{q(w)}\]

\[ \; =  \sum\limits_{i = 1}^n \CS (\lm_i w_i, y) \frac{q(\lm_i w_i)}{q(w)} = \sum\limits_{i = 1}^n \CS (w_i, y) \frac{q (\lm_i w_i)}{q(w)}. \]
\end{proof}

\begin{lem}\label{lem:8.2} Given $x_1, \dots, x_m$, $y_1, \dots, y_m$, $\al_1, \dots, \al_m$ in $eR$ such that  that
\begin{equation}
 \al_i x_i < \al_i y_i \qquad \mbox{for} \; 1 \leq i \leq m,  \tag{$*$} \end{equation}
then
\begin{equation} \sum\limits_{i = 1}^m \al_i x_i < \sum\limits_{i = 1}^m \al_i y_i. \tag{$**$} \end{equation}
\end{lem}
\begin{proof}
 Choose $r \in \{ 1, \dots, m \}$ such that $\al_r x_r = \Max\limits_{1 \leq i \leq m} \{\al_i x_i\}$,
then
\[ \sum\limits_{i = 1}^m \al_i x_i = \al_r x_r < \al_r y_r \leq \sum\limits_{i=1}^m \al_i y_i. \]
\end{proof}

\begin{rem}\label{rem:8.3} If $(*)$ holds with the strict inequality  $<$ replaced by the weak inequality  $\leq$ (respectively the equality sign =) everywhere, then $(**)$ holds with $\leq$ (respectively =) everywhere. This is trivial. \end{rem}

We are ready to prove a convexity lemma for linear $\CS$-inequalities, which will play a central role in the rest of the paper.

\begin{lem}[CS-Convexity Lemma]\label{lem:8.4} Let $\Box$ be one of the symbols $<, \leq, =$. Given rays $W_1, \dots, W_m$, $Y_1, \dots, Y_n$ in $V$ and scalars $\gm_1, \dots, \gm_n$, $\delta_1, \dots, \delta_n$ in $eR$ such that \begin{equation} \sum\limits_{j = 1}^n \gm_j \CS (W_i, Y_j) \; \ds\Box \; \sum\limits_{j = 1}^n \delta_j \CS (W_i, Y_j), \qquad \text{for } i = 1, \dots, m. \tag{$*$} \end{equation}
Then,  for every $W \in \conv (W_1, \dots, W_m)$,
\begin{equation} \sum\limits_{j = 1}^n \gm_j \CS (W, Y_j) \; \ds \Box \; \sum\limits_{j = 1}^n \delta_i \CS (W, Y_j). \tag{$**$} \end{equation}
\end{lem}
\begin{proof}
 We verify the assertion  for $<$. By Lemma \ref{lem:8.1} we have scalars $\al_1, \dots, \al_m \in \tG$ such that
\[ \CS (W, Y_j) = \sum\limits_{i = 1}^m \al_i \CS (W_i, Y_j) \]
for $j = 1, \dots, n$. Using Lemma \ref{lem:8.2} and the fact that the $\al_i$ are units of $eR$,  we obtain
\begin{align*}
 \sum\limits_{j = 1}^n \gm_j \CS (W, Y_j) & = \sum\limits_{j = 1}^n \gm_j \sum\limits_{j = 1}^m \al_i \CS (W_i, Y_j) = \sum\limits_{i = 1}^m \al_i \sum\limits_{j = 1}^n \gm_j \CS (W_i, Y_j) \\ & <
 \sum\limits_{i = 1}^m \al_i \sum\limits_{j = 1}^n \delta_j \CS (W_i, Y_j) = \sum\limits_{j = 1}^n \delta_j \sum\limits_{i = 1}^m \al_i \CS (W_i, Y_j) = \sum\limits_{j = 1}^n \delta_j \CS (W, Y_j).
\end{align*}
For the other signs $\leq,$ $ =$ the argument is analogous, using  Remark~\ref{rem:8.3} instead of Lemma~\ref{lem:8.2}. (Here it does not matter that the $\al_i$'s are units of $eR$.) \end{proof}

 We start with an application of the CS-Convexity Lemma~\ref{lem:8.4} upon subsets of the ray space $\Ray (V)$ related to the CS-profile types introduced in \S\ref{sec:7}.

\begin{defn}\label{def:8.5} Given a pair $(Y_1, Y_2)$ of different rays in $V$ and a basic type $T$ (as listed in Tables \ref{table:7.3}, \ref{table:7.4} and Scholium \ref{schol:7.5}), we define the $T$-\textbf{locus of} $\overrightarrow{[Y_1, Y_2]}$ as the set of all rays~ $W$ in $V$ which have a $W$-profile of type $T$ on $\overrightarrow{[Y_1, Y_2]}$, and denote this subset of $\Ray (V)$ by $\Loc_T (Y_1, Y_2)$. \end{defn}

 It is understood that, if $T$ is defined under the condition, say, $\CS (Y_1, Y_2) > e$,\footnote{Taken up to interchanging $Y_1, Y_2$, the type $T$ is listed in Table \ref{table:7.3}.} the sentence ($\CS (Y_1, Y_2) > e$) is part of the sentence $T$, and thus  $\Loc_T (Y_1, Y_2) = \emptyset$ when actually $\CS (Y_1, Y_2) \leq e$. So, using standard notation from logic, we may write
\begin{equation}\label{eq:8.3} \Loc_T (Y_1, Y_2) = \{ W \in \Ray (V) \ds \vert W \models T \} \end{equation}
for any pair $(Y_1, Y_2)$ of different rays in $V$ and any basic type $T$. We call these subsets $\Loc_T (Y_1, Y_2)$ of $\Ray (V)$ the \textbf{basic loci} of the interval $[Y_1, Y_2]$.

\begin{thm}\label{thm:8.6} The family of nonempty basic loci of any interval $[Y_1, Y_2]$ in $\Ray (V)$ is a partition of $\Ray (V)$ into convex subsets.
\end{thm}\begin{proof}
 a) The family of basic loci of $[Y_1, Y_2]$ is a partition of $\Ray (V)$, as for a given $W \in \Ray (V)$ the function $\CS (W, \udscr)$ on $\overrightarrow{[Y_1, Y_2]}$ has a profile of type $T$ for exactly one basic $T$.\pSkip
b) Given a basic type $T$ for $\overrightarrow{[Y_1, Y_2]}$ it is a straightforward consequence of Lemma~\ref{lem:8.4} (with $m = 2$) that $\Loc_T (Y_1, Y_2)$ is convex. We show this in the case that $\CS (Y_1, Y_2) > e$ and $T$ is the condition $B$ in  Table~\ref{table:7.3}. Let $W_1, W_2 \in \Loc_B (Y_1, Y_2)$ and $W \in [W_1, W_2]$ be given. Then
\begin{equation}
  0 < \frac{\CS (W_i, Y_2)}{\CS (Y_1, Y_2)} < \CS (W_i, Y_1) < \CS (W_i, Y_2), \qquad \text{for } i = 1,2. \tag{$*$}
\end{equation}
 Applying the CS-Convexity Lemma~\ref{lem:8.4} to the inequality on the right (with $n = 1$, $\gm_1 = \delta_1 = ~ e$),  we obtain
\[ \CS (W, Y_1) < \CS (W, Y_2), \]
and thus also $0 < \frac{\CS (W, Y_2)}{\CS (Y_1, Y_2)}$. Applying the lemma to the inequality in the middle of ($\ast$) (with $n = 1$, $\gm_1 = \CS (Y_1, Y_2)^{-1}$, $\delta_1 = e$) we obtain
\[ \frac{\CS (W, Y_2)}{\CS (Y_1, Y_2)} < \CS (W, Y_1). \]
This proves condition $B$ for $(W, Y_1, Y_2)$. \end{proof}

\begin{defn}\label{def:8.7} Given a composite type $U$ of CS-profiles (cf. \S\ref{sec:7}), in analogy to \eqref{eq:8.3} we define the \textbf{$U$-locus} of $\overrightarrow{[Y_1, Y_2]}$ as
\begin{equation}\label{eq:8.4} \Loc_U (Y_1, Y_2) : = \{ W \in \Ray (V) \ds \vert W \models U \}. \end{equation}
\end{defn}

\begin{thm}\label{thm:8.8} The subset $\Loc_U (Y_1, Y_2)$ is convex in $\Ray (V)$ for every pair $(Y_1, Y_2)$ of different rays in $V$.
\end{thm}

\begin{proof}
 $U$ is by definition a relaxation of a basic type $T$, and so the inequalities in $U$ are obtained by replacing in $T$ the strict inequality $<$ by $\leq$ at several places. The CS-Convexity Lemma~\ref{lem:8.4}, taken now  for weak inequalities, gives the claim.
\end{proof}

\begin{rem}\label{rem:8.9} As stated in \eqref{eq:7.1},  $U$ is a disjunction of finitely many basic types,
\[ U = T_1 \vee T_2 \vee \dots \vee T_r. \]
It follows from \eqref{eq:8.4} that
\begin{equation}\label{eq:8.5} \Loc_U (Y_1, Y_2) = \bigcup\limits_{i = 1}^r \Loc_{T_i} (Y_1, Y_2). \end{equation}
 \end{rem}

\section{Downsets of restricted $\QL$-stars}\label{sec:9}

Recall that the \textbf{QL-star} $\QL(X)$ of a ray $X$ (with respect to $q$) is the set of all $Y \in \Ray(V)$ for which the pair $(X,Y)$ is quasilinear; equivalently, the interval $[X,Y]$ is quasilinear \cite[Definition 4.5]{Quasilinear}. The QL-stars determine the quasilinear behavior of $q$ on the ray space.

Given a $\QL$-star $\QL(X)$ we investigate the \textbf{downset of} $\QL(X)$, i.e., the set of all $\QL$-stars $\QL(Y) \subset \QL(X)$, partially ordered by inclusion. We translate this problem into the language of rays by considering the downset $\{ Y \in \Ray (V) \ds \vert Y \preceq_{\QL} X \}$ with respect to the quasiordering $\preceq_{\QL}$ given in~(4.1). (Recall from \cite[\S5]{Quasilinear} that a $\QL$-star $\QL(Y)$ corresponds uniquely to the equivalence class of $Y$ with respect  to $\sim_{\QL}$.)

More generally fixing a nonempty set $D \subset \Ray (V)$,  for any $X \in \Ray (V)$ we define
\[ \QL_D (X) : = \QL (X) \cap D, \]
and explore the  downsets of this family of sets, ordered by inclusion. To do so,  without extra costs, we  pass to a coarsening $\preceq_D$ of the quasiordering $\preceq_{\QL}$, defined as
\[ X \preceq_D X' \dss \Iff \QL_D (X) \subset \QL_D (X'), \]
with associated equivalence relation
\[ X \sim_D X' \dss \Iff \QL_D (X) = \QL_D (X'). \]
We call the set $\QL_D (X)$ the \textbf{restriction of the $\QL$-star of $X$ to $D$.}

We use the monotone $\CS$-profiles on an interval $[Y_1, Y_2]$, whenever they occur, to investigate these downsets. Yet,  we need a criterion  for quasilinearity of pairs of anisotropic rays in \cite{QF2} (under a stronger assumption  than before on $R$) which for the present paper reads as:

\begin{thm}[{\cite[Theorems 6.7 and 6.11]{QF2}}]\label{thm:9.1} Assume that $R$ is a \textbf{nontrivial tangible supersemifield}, i.e., $R$ is a supertropical semiring in which both $\tG = eR \setminus \{ 0 \}$ and $\mathcal{T} = R \setminus (eR)$ are abelian groups under multiplication, $e \mathcal{T} = \tG$, and $\tG \ne \{ e \}$. Then a pair $(W, Z)$ of anisotropic rays on $V$ is quasilinear iff either $\CS (W, Z) \leq e$, or $\tG$ is discrete, $\CS (W, Z) = c_0$\footnote{$c_0$ denotes the smallest element $> e$ in $\tG$. It exists since $\tG$ is discrete.}, and both $W$ and $Z$ are $g$-isotropic (i.e., $q(w)$, $q(z) \in \tG$ for all $w \in W$, $z \in Z$), saying in the latter case that $(W, Z)$ is \textbf{exotic quasilinear}.
\end{thm}

\begin{thm}\label{thm:9.2}  Assume that $R$ is a nontrivial tangible supersemifield and that the quadratic form $q$ is anisotropic on $V$. Given a (nonempty) subset $D$ of $\Ray (V)$ let $X, Y_1, Y_2$ be rays in~ $V$ with $Y_1 \preceq_D X$ and $Y_2 \preceq_D X$, and assume that the $\CS$-profile of every $W \in D$ on $[Y_1, Y_2]$ is monotone. Then the following holds.
\begin{itemize}\dispace
\item[i)] If $\tG$ is dense, then $Y \preceq_D X$ for every $Y \in [Y_1, Y_2]$.
\item[ii)] If $\tG$ is discrete, then $Y \preceq_D X$ for every $g$-anisotropic $Y \in [Y_1, Y_2]$.
\item[iii)] If $\tG$ is discrete and at least one of the rays $Y_1, Y_2$ is $g$-isotropic, then $Y \preceq_D X$ for every $Y \in [Y_1, Y_2]$.
\end{itemize}
\end{thm}\begin{proof}
 The study in \S\ref{sec:6} of the functions $\CS (W, \udscr)$ on closed intervals reveals that for a given $W \in \Ray (V)$ this function is monotone, i.e., increasing or decreasing on $\overrightarrow{[Y_1, Y_2]}$ iff
\begin{equation}\label{eq:9.1} \forall Z \in [Y_1, Y_2]: \quad  \CS (W, Z) \geq \min (\CS (W, Y_1), \CS (W, Y_2)). \end{equation}
In the following we only rely on this property.

Let $Y \in [Y_1, Y_2]$ and $W \in \QL (Y) \cap D$ be given. We  prove that $W \in \QL (X)$ under  conditions i) -- iii), and then will be done. \\
a) Suppose that $\CS (W, Y) \leq e$. Since the $W$-profile on $[Y_1, Y_2]$ is monotonic, we conclude from \eqref{eq:9.1} that $\CS (W, Y_i) \leq e$ for $i = 1$ or 2, whence $W \in \QL_D (Y_i)$, and so $W \in \QL_D (X)$ since $Y_i \preceq_D X$. This  settles claim i) of the theorem, as well as the other claims in the case $\CS (W, Y) \leq e$.\pSkip
b) There remains the case that $eR$ is discrete and $\CS (W, Y) = c_0$. The pair $(W, Y)$ is exotic quasilinear, since $W \in \QL(Y)$, and thus $W$ and $Y$ are $g$-isotropic. But, under the assumption in  ii) of the theorem this does not hold, whence this case cannot occur.\pSkip
c) We are left with a proof of part iii). If $\CS (W, Y_1) \leq e$ or $\CS (W, Y_2) \leq e$, the same argument as in a) gives that $W \in \QL (Y_1)$ or $W \in \QL (Y_2)$, and so $W \in \QL (X)$. Henceforth we assume that $\CS (W, Y_1) = \CS (W, Y_2) = c_0$. If, say, $Y_1$ is $g$-isotropic, then the pair $(W, Y_1)$ is exotic quasilinear, whence $W \in \QL (Y_1)$, and so $W \in \QL (X)$, as desired. 
\end{proof}

 We list several cases where a given $\CS$-profile is monotonic, now using in detail the profile analysis from \S\ref{sec:6}. (Here it suffices to assume that $eR$ is a semifield.)

\begin{schol}\label{schol:9.3} As before we assume that $q$ is anisotropic on $V$.
\begin{itemize}\dispace
\item[a)] If $[Y_1, Y_2]$ is $\nu$-quasilinear, then $[Y_1, Y_2]$ has a monotonic $W$-profile for every $W \in \Ray (V)$.
\item[b)] If $[Y_1, Y_2]$ is $\nu$-excessive, with critical rays $Y_{12}$ (near $Y_1$) and $Y_{21}$ (near $Y_2$), then both $[Y_1, Y_{12}]$ and $[Y_{21}, Y_2]$ have a monotonic $W$-profile for every $W$.
\item[c)] Given $W \in \Ray (V)$ let $M = M (W, Y_1, Y_2)$ denote the $W$-median of $[Y_1, Y_2]$. Assume that the $W$-profile of $[Y_1, Y_2]$ is not monotone. Then $[Y_1, M]$ and $[M, Y_2]$ are the maximal closed subintervals of $[Y_1, Y_2]$ with a monotone $W$-profile.
\item[d)] Of course, if $[Y_1, Y_2]$ has a monotonic $W$-profile for a given ray $W$, then the same holds for every closed subinterval of $[Y_1, Y_2]$.
\end{itemize}
\end{schol}

We next  search for rays  $Z \preceq_D Y$ in a given interval $[X, Y]$. Here and elsewhere it is convenient to extend our notion of $\QL$-stars and related objects from rays to vectors in a trivial way as follows (assuming only that $eR$ is a semifield).

\begin{notation}\label{notat:9.4}
   Given $\mbox{x, y} \in V \setminus \{ 0 \}$ we define
$$ \QL (x) : = \QL (\ray (x)), $$
and set
\[ \mbox{x} \preceq_{\QL} \mbox{y}  \dss \Iff \ray (x) \preceq_{\QL} \ray (y), \]
equivalently
\[ \mbox{x} \preceq_{\QL} \mbox{y} \dss \Iff \QL (x) \subset \QL (y). \]
Consequently we define
\[ \begin{array}{lclll}
\mbox{x} \sim_{\QL} \mbox{y} &  \Iff &  \QL (x) = \QL (y) & \Iff & \ray (x) \sim_{\QL} \ray (y).
\end{array} \]
More generally, if a set $D \subset \Ray (V)$ is given, we put
\[ \begin{array}{lcl}
\QL_D (x) &  = & \QL_D (\ray (x)) ,\\[2mm]
\mbox{x} \preceq_D \mbox{y} &  \Iff & \QL_D (x) \subset \QL_D (y), \\[2mm]
\mbox{x} \sim_D \mbox{y} &   \Iff & \QL_D (x) = \QL_D (y).
\end{array} \]
\end{notation}

 As before we assume that $q$ is anisotropic on $V$, and that $R$ is a nontrivial tangible supersemifield.

\begin{lem}\label{lem:9.5} Let $x, y \in V \setminus \{ 0 \}$ and assume that $q(x+y)\cong_{\nu} q(y)$.
\begin{enumerate}
  \item[a)] For any $w \in V \setminus \{ 0 \}$
\[ \CS (w, y) \leq \CS (w, x+y). \]
\item[b)] If $q(y) \in \tG$, then $q(x+y) = q(y)$.

\end{enumerate}

\end{lem}
\begin{proof}
 a): As in the proof of Lemma \ref{lem:8.1}, we have
\[ \CS (w, x+y) = \CS (w, x) \frac{q(x)}{q(x+y)} + \CS (w, y) \frac{q(y)}{q (x+y)} \]
 which implies that $\CS (w, y) \leq \CS (w, x+y)$, since $\frac{q(y)}{q(x+y)} \cong_{\nu} e$.
\pSkip
b): A priori we have $q(x+y) = q(x) + q(y) + b (x,y)$, and so $q(y) \leq q(x+y)$ in the minimal ordering on $R$. Since by assumption $eq (y) = eq (x+y)$ and $q(y) \in \tG$, it follows that $ q(x+y) = q (y)$.\footnote{Note that $\al \cong_{\nu} \bt $, $\al \in \tG \Rightarrow \bt  \leq \al$ for any $\al, \bt  \in R$.} \end{proof}

\begin{thm}\label{thm:9.6} Assume  that $q (x+y) \cong_{\nu} q(y)$ for given $x, y \in V \setminus \{ 0 \}$. When  $\tG$ is discrete, assume also that $q(y) \in \tG$.
\begin{enumerate}\dispace
  \item[a)]  $x+y \preceq_{\QL} y$ (and so $x+y \preceq_D y$ for any $D \subset \Ray (V)$).
\item[b)] The interval $[ \ray (x+y), \ray (y)]$ has a monotone $W$-profile for every ray $W$ in $V$. \end{enumerate}
\end{thm}

\begin{proof}
 a): Let $w \in V \setminus \{ 0 \}$ be given with $w \in \QL (x+y)$, we verify that $w \in \QL (y)$. By Lemma~\ref{lem:9.5}.a we know that
\begin{equation}
 \CS (w, y) \leq \CS (w, x+y). \tag{$*$}\end{equation}
Assume first that $\tG$ is dense, then $\CS (w, x+y) \leq e$, since $w \in \QL (x+y)$. From  ($\ast$) it follows  that $\CS (w, y) \leq e$, and so $w \in \QL (y)$.

Assume next  that $\tG$ is discrete. If $\CS (w, y)\leq e$, then certainly $w \in \QL (y)$. There remains the case that
\begin{equation} \CS (w, y) \geq c_0. \tag{$**$}\end{equation}
Now ($\ast$) tells us that $\CS (w, x+y) \geq c_0$. Since $w \in \QL (x+y)$, we conclude by Theorem~\ref{thm:9.1} that $\CS (w, x+y) = c_0$ and $q(w) \in \tG$, $q(x+y) \in \tG$. From ($\ast$) and ($\ast \ast$) we infer that $\CS (w, y) = c_0$. Thus, again by Theorem~\ref{thm:9.1}, $w \in \QL (y)$.\pSkip
b): Let $Z : = \ray (x+y)$, $Y : = \ray (y)$. We proved that $Z \preceq_{\QL} Y$, i.e., $\QL (Z) \subset \QL (Y)$. This implies that $Z \in \QL (Y)$, i.e., that $[Y, Z]$ is quasilinear. All the more $[Y, Z]$ is  $\nu$-quasilinear, and Scholium~\ref{schol:9.3}.a confirms that it has a monotone $W$-profile for any $W \in \Ray (V)$. \end{proof}

\section{The medians of a closed ray-interval}\label{sec:10}

For a short period we only assume that $(q, b)$ is a quadratic pair on an $R$-module $V$ where~ $R$ is a supertropical semiring without zero divisors, and $\lm x \neq 0$ for nonzero $\lm \in R$ and all nonzero $x\in V $,  cf. \cite[\S6]{QF2}. Recall that two vectors $x, x' \in V$ are said to be \textbf{ray-equivalent}, written $x \sim_r x'$, if there exist scalars $\lm, \lm' \in R \setminus \{ 0 \} = \tG$ with $\lm x = \lm' x'$; the ray-equivalence class of $x \ne 0 $ is denoted $\ray (x)$. We introduce a map
\[ m : V \times V \times V \longrightarrow V \]
by the rule
  \begin{equation}\label{eq:10.1}
 m (w, x, y) : = b (w, y) x + b (w, x) y.
 \end{equation}
Obviously this map is $R$-trilinear, and so is compatible with ray-equivalence, i.e., if $w \sim_r w'$, $x \sim_r x'$, $y \sim_r y'$ then $m (w, x, y) \sim_r m (w', x', y')$. Thus for any three rays $W = \ray (w)$, $X = \ray (x)$, $Y = \ray (y)$, where \textbf{not} $b (w, x) = b (w, y) = 0$, we obtain a well defined ray
 \begin{equation}\label{eq:10.2} M(W, X, Y) = \ray ( m (w, x, y) )\in [X, Y]. \end{equation}
In fact, at least one of the vectors $b(w, y) x$, $b (w, x)y$ is not zero, and so $$m (w, x, y) \in (R x + Ry) \setminus \{ 0 \}.$$

Here the notion of the ``polar'' of a subset $C$ of $\Ray (V)$ comes into play, defined as follows.
\begin{defn}\label{def:10.1}
Given a (nonempty) subset $C$ of $\Ray (V)$, the \textbf{polar} $C^{\perp}$ is the set of all $W \in \Ray (V)$ with $b (w, x) = 0$ for all $w \in W$ and $x \in V \setminus \{ 0 \}$ with $\ray (x) \in C$.
\end{defn}
It is immediate from Definition \ref{def:10.1} that any polar $C^{\perp}$ is convex and also that a set $C \subset \Ray( V)$ and its convex hull $\conv (C)$ have the same polar,
 \begin{equation}\label{eq:10.3} C^{\perp} = \conv (C)^{\perp}. \end{equation}
Thus it suffices most often to consider polars of convex sets. If $C$ is convex, then we can characterize both $C$ and $C^{\perp}$ by the ray closed subsets $\tlC$ and $(C^{\perp})^{\sim}$ of $V \setminus \{ 0 \}$,  associated to~ $C$ and $C^{\perp}$ (cf. \cite[Notation 2.4]{Quasilinear}) apparently as follows:
 \begin{equation}\label{eq:10.4} (C^{\perp})^{\sim} = \{ w \in V \setminus \{ 0 \} \ds\vert b (w, x) = 0 \ \mbox{for every} \; x \in \tlC \}. \end{equation}
For the ray-closed submodules
\[ T = \tlC \cup \{ 0 \}, \quad U = (C^{\perp})^{\sim} \cup \{ 0 \} \]
(cf. \cite[Remark 2.6]{Quasilinear}) we conclude that
 \begin{equation}\label{eq:10.5} U = \{ w \in V \ds\vert b (w, t) = 0 \ \mbox{for every} \; t \in T \}.\end{equation}
We now take a look at the complement of a polar in $\Ray (V)$.

\begin{prop}\label{prop:10.2}
 Assume that $C$ is a subset of $\Ray (V)$ and that $W_1$ is a ray in  $V$ where  $W_1 \not\in C^{\perp}$. Then for any $W_2 \in \Ray (V)$ the half-open interval $[W_1, W_2[$ \footnote{ The overall assumption  that $eR$ is a semifield is not necessary, cf. \cite[Definition 7.5]{QF2}.} is disjoint from~ $C^{\perp}$. \end{prop}

\begin{proof}Writing $X^{\perp} : = \{ X \}^{\perp}$ for $X \in \Ray (V)$ it is obvious that
 \begin{equation}\label{eq:10.6} C^{\perp} = \bigcap\limits_{X \in C} X^{\perp}.\end{equation}
Thus, there exists some $X \in C$ such that  $W_1 \not\in X^{\perp}$. Picking vectors $w_1 \in W_1$, $w_2 \in W_2$, $x \in X$, then $b(w_1, x) \ne 0$, which implies that for any scalars $\lm_1 \in R \setminus \{ 0 \}$, $\lm_2 \in R$
\[ b (\lm_1 w_1 + \lm_2 w_2, x) = \lm_1 b (w_1, x) + \lm_2 b (w_2, x) \ne 0. \]
Since such vectors $\lm_1 w_1 + \lm_2 w_2$ represent all rays in $[W_1, W_2 [$ , we conclude that $[W_1, W_2 [ \, \cap \, X^{\perp} \linebreak = \emptyset$. All the more $[W_1, W_2 [  \, \cap \, C^{\perp} = \emptyset$. \end{proof}

\begin{cor}
\label{cor:10.3.} Given any subset $C$ of $\Ray (V)$, both sets $C^{\perp}$ and $\Ray (V) \setminus C^{\perp}$ are convex.\end{cor}

\begin{proof}
 Convexity of $C^{\perp}$ had been observed above. The convexity of $\Ray (V) \setminus C^{\perp}$ follows from Proposition~\ref{prop:10.2} by taking  $W_2$ in $\Ray (V) \setminus C^{\perp}$. \end{proof}

We are ready for the key definition of this section.

\begin{defn}
\label{def:10.4} Given three rays $W, X, Y$ in $V$ with $W \not\in [X, Y]^{\perp} = \{ X, Y \}^{\perp}$,  the ray $M (W, X, Y)$ from  \eqref{eq:10.2} is called the $W$-\textbf{median} of the pair $(X, Y)$.\end{defn}

We denote this ray most often by $M_W (X, Y)$ instead of $M (W, X, Y)$. This notation emphasizes the fact, obvious from \eqref{eq:10.1} and \eqref{eq:10.2}, that
 \begin{equation}\label{eq:10.7} M_W (X, Y) = M_W (Y, X). \end{equation}
The assignment $W \mapsto M_W (X, Y)$ has convexity properties as follows.

\begin{thm}
\label{thm:10.5} Assume that $W_1, W_2, X, Y$ are rays in $V$ with $W_1 \not\in \{ X, Y \}^{\perp}$, $W_2 \not\in \{ X, Y \}^{\perp}$.
\begin{enumerate}\dispace
  \item[a)]  $[W_1, W_2] \cap \{ X, Y \}^{\perp} = \emptyset$ and so $M_W (X, Y)$ is defined for every $W \in [W_1, W_2 ]$.
\item[b)] For any $W \in [W_1, W_2 ]$
 \begin{equation}\label{eq:10.8} M_W (X, Y) \in [M_{W_1} (X, Y), M_{W_2} (X, Y) ]. \end{equation}
\end{enumerate}
\end{thm}
\begin{proof}
 a): Clear from Proposition 10.2, applied to $C = \{ X, Y \}$.\pSkip
b): We may assume that $W \ne W_1$, $W \ne W_2$. For vectors $w_1 \in W_1$, $w_2 \in W_2$, $x \in X$, $y \in Y$. there exist scalars $\lm_1, \lm_2 \in R \setminus \{ 0 \}$ such that  $w = \lm_1 w_1 + \lm_2 w_2 \in W$. Then $m (w, x, y) = \lm_1 m (w_1, x, y) + \lm_2 m (w_2, x, y)$, and so $M_W (X, Y) = \ray (\lm_1 m (w_1, x, y) + \lm_2 m (w_2, x, y) ) \in [ M_{W_1} (X, Y), M_{W_2} (X, Y) )]$. \end{proof}

We state an immediate consequence of this theorem.

\begin{cor}
\label{cor:10.6} Let $X, Y$ be (different) rays in $V$. Assume that $S$ is a convex subset of the closed interval $[X, Y]$. Then the set of all rays $W$ in $V$ with $W \not\in \{ X, Y \}^{\perp}$ and $M_W (X, Y) \in ~S$ is convex in $\Ray (V)$. \end{cor}

Assume now again, as mostly from \S 6 onward, that $eR$ is a semifield. Then we know from \cite[Theorem 8.8]{QF2}, that the ``border rays'' $X, Y$ of $[X, Y]$ are uniquely determined by $[X, Y]$ up to permutation. Thus we are entitled to call $M_W (X, Y)$ the $W$-\textbf{median of} $[X, Y]$.

The $W$-median $M_W (X, Y)$ have already appeared in our analysis of the function $f(\lm) = CS (\veps_1, \veps_2 + \lm \veps_3)$ in \S\ref{sec:6} under the labeling $X_1, X_2, X_3$ instead of $W, X, Y$, with $\veps_i \in X_i$ and $\veps_1, \veps_2, \veps_3$ instead of $w, x, y$. From the analysis in \S\ref{sec:6} we can read off the following important facts about $M_W (X, Y)$.

\begin{thm}
\label{thm:10.7} Assume that $eR$ is a semifield and $W, X, Y$ are rays in $V$ where  $W \not\in [X, Y]^{\perp} = \{ X, Y \}^{\perp}$, so that the $W$-median $M_W (X, Y)$ is defined. Assume also that the rays $W, X, Y$ are anisotropic (i.e., none of the sets $q(W)$, $q(X)$, $q(Y)$ contains zero), so that the $\CS$-ratio $\CS (W, Z)$ is defined for every $Z \in [X, Y]$.
\begin{enumerate}\dispace
  \item[a)] The function
\[ Z \longmapsto \CS (W, Z), \quad  [X, Y] \longrightarrow eR \]
attains its minimal value at
\[ M : = M_W (X, Y). \]

\item[b)] If the function $\CS (W, \udscr)$ is monotone on $[X, Y]$ (with respect  to $\leq_X$), then
 \begin{equation}\label{eq:10.8b} \CS (W, M) = \min (\CS (W, X), \CS (W, Y)) \end{equation}
is this minimal value.
\item[c)] $\CS (W, \udscr)$ is monotone on $[X, Y]$ iff either $\CS (X, Y) \leq e$ or $\CS (X, Y) > e$ and $M \not\in \, ] X^*, Y^* [$ where  $X^*, Y^*$ denote  the critical rays of $[X, Y]$ near $X$ and $Y$ respectively (cf. \cite[\S9]{QF2}).

\item[d)]  If $\CS (W, \udscr)$ is not monotone (and so $\CS (X, Y) > e$, $M \in \,  ] X^*, Y^*[$ ),  the minimal value of $\CS (W, \udscr)$ on $[X, Y]$ is
 \begin{equation}\label{eq:10.9} \CS (W, M) = \sqrt{\frac{\CS (W, X) \CS (W, Y)}{\CS (X, Y)}} < \min (\CS (W, X), \CS (W, Y)). \end{equation}
Furthermore,  $Z = M$ is the \textbf{only ray in} $[X, Y]$ \textbf{where the minimum is attained.}
\end{enumerate}
\end{thm}
\begin{proof} As pointed out, all proofs have been done in \S\ref{sec:6}. The ray $M$ corresponds to $\lm = \xi : = \frac{\al_{12}}{\al_{13}}$, cf. \eqref{eq:6.12} and \eqref{eq:6.13}. Claim a) is evident by the argument following \eqref{eq:6.12}. Claim b) then is trivial. The other two claims c) and d) follow from the description of the monotonic behavior of $f(\lm)$ in Proposition~\ref{prop:6.1} and Theorems~\ref{thm:6.2} and \ref{thm:6.3}. Note that there, when $\CS (X_2, X_3) > e$, the critical rays of $[X_2, X_3]$ near $X_2$ and $X_3$ (as defined in \cite[\S9]{QF2}) correspond to $\lm = \frac{\al_2}{\al_{23}}$ and $\lm = \frac{\al_{23}}{\al_3}$ respectively, cf. \eqref{eq:6.20}. \end{proof}

\begin{rem}
\label{rem:10.8} In the terminology of \S\ref{sec:7} the function $\CS (W,\udscr)$ on $[X, Y]$ is not monotone iff the $\CS$-profile of $W$ on $[X, Y]$ is of type $B, B'$ or $\partial B \ (= \partial B')$, cf. Tables~\ref{table:7.3} and  \ref{table:7.4} and Scholium~\ref{schol:7.5}. The ray $W$ is in the polar $\{ X, Y \}^{\perp}$ iff $\CS (W, \udscr)$ is zero everywhere on $[X, Y]$, which means that the $\CS$-profile of $W$ on $[X, Y]$ is of type $\partial D$ or $\partial E$. \end{rem}

\section{On maxima and  minima of $\CS (W,-)$ on finitely generated convex sets in the ray space}\label{sec:11}

We call a convex subset $C$ of $\Ray (V)$ \textbf{finitely generated}, if $C$ is the convex hull of a finite set of rays $\{ Y_1, \dots, Y_n \}$, and call $\{ Y_1, \dots, Y_n \}$ a \textbf{set of generators} of $C$. Note that then the sum $(Y_i)_0 + \dots + (Y_n)_0$ of the submodules $(Y_i)_0 = Y_i \cup \{ 0 \}$ of $V$ is the ray-closed submodule $U$ of $V$ with $\Ray (U) = C$, cf. \S\ref{sec:1}.

Assume again, as previously, that the ghost ideal $eR$ of the supertropical semiring $R$ is a semifield and $(q, b)$ is a quadratic pair on the $R$-module $V$. Assume also that $q$ is anisotropic on $V$. (Otherwise we replace $V$ by $V_{\an}$). Then the $\CS$-ratio $\CS (W, Z)$ is defined for any two rays $W, Z$ in $V$. Assume finally that $C$ is a finitely generated convex subset of $\Ray (V)$ and $(Y_1, \dots, Y_n)$ is a fixed sequence of generators of $C$.

Given a ray $W$ on $V$, we enquire  whether the function $\CS (W, \udscr)$ on $C$ has a minimal value, and then, where on $C$ this  minimal value is attained.\footnote{We leave the important problem aside, whether $C$ has a \textit{unique minimal} set of generators. It would take us too far afield.}
To have a precise hold at the function $Y \mapsto \CS (W, Y)$, $C \rightarrow eR$, we use the following notation. Given vectors $y_i \in e Y_i$ $(1 \leq i \leq n)$, a ray  $W \in \Ray (V)$, and vectors $w \in eW$, $y \in e Y$, $Y \in C$, then $Y_i = \ray (y_i)$ and  $W = \ray (w)$, $Y = \ray (y)$. We have a presentation
 \begin{equation}\label{eq:11.1} y = \lm_1 y_1 + \dots + \lm_n y_n \end{equation}
with a sequence $(\lm_1, \dots, \lm_n) \in e R^n$, not all $\lm_i = 0$.

Let $\al_i : = q (y_i)$, $\beta_{ij} : = b (y_i, y_j)$. For $i, j \in \{ 1, \dots, n \}$ we define
\[ \gamma_{ij} = \left\{ \begin{array}{lcl}
\al_i & \mbox{if} & i = j, \\[1mm]
\beta_{ij} & \mbox{if} & i \ne j, \\
\end{array} \right. \]
and  write
 \begin{equation}\label{eq:11.2} q (y) = \sum\limits_{1 \leq i \leq j \leq n} \gamma_{ij} \lm_i \lm_j. \end{equation}
By a computation as in the proof of Lemma~\ref{lem:8.1} we obtain a useful formula for $\CS (W, Y) = \CS (w, y)$, interchanging there the arguments $w, y$, namely
\[ \CS (w, y) = \sum\limits_{i = 1}^n \CS (w, \lm_i y_i) \frac{q (\lm_i y_i)}{q (y)}, \]
and so
 \begin{equation}\label{eq:11.3} \CS (W, Y) = \sum\limits_{i = 1}^n \CS (W, Y_i) \frac{\lm_i^2 \al_i}{q (y)}. \end{equation}

We now are ready for the central result of this section.

\begin{thm}\label{thm:11.1} Let $W, Y_1, \dots, Y_n \in \Ray (V)$ be given and $C : = \conv (Y_1, \dots, Y_n)$. Then the $eR$-valued function $\CS (W, \udscr)$ on $C$ has a minimum. It is attained at $Y_r$ for some $r \in \{ 1, \dots, n \}$ or at the $W$-median $M_W (Y_r, Y_s)$ of some interval $[Y_r, Y_s], 1 \leq r \leq s \leq n$, on which $\CS (W, \udscr)$ is not monotone.
\end{thm}
\begin{proof} Without loss of generality we assume that $\CS (W, Y_1) \leq \CS (W, Y_i)$ for $2 \leq i \leq n$. We distinguish two cases.

\begin{description}\dispace
  \item[\textbf{Case A}] $\CS (W, Y_1) \leq \CS (W, Y)$ for every $Y \in C$. Now the minimum is attained at ~$Y_1$.
  \item[\textbf{Case B}] There exists some $Y \in C$ with $\CS (W, Y) < \CS (W, Y_1)$.
\end{description}
We choose $(\lm_1, \dots, \lm_n) \in eR^n$ such that \eqref{eq:11.1} and  \eqref{eq:11.2} hold for $W$ and $Y$, and so does \eqref{eq:11.3}. We further choose a dominant term $\gamma_{r, s} \lm_r \lm_s$, $r \leq s$, in the sum on the right of \eqref{eq:11.2}. Then $q (y) = \gamma_{rs} \lm_r \lm_s$, and so
\[ \CS (W, Y) = \sum\limits_{i = 1}^n \CS (W, Y_i) \frac{\lm_i^2 \al_i}{\gamma_{rs} \lm_r \lm_s}. \]
Clearly
\begin{equation}\label{eq:str} \sum\limits_{i = r}^s \CS (W, Y_i) \frac{\al_i \lm_i^2}{\gamma_{rs} \lm_r \lm_s} \ds \leq \CS (W, Y) \ds  < \CS (W, Y_1). \tag{$*$}
\end{equation}

Suppose that $r = s$. Then we obtain that $\CS (W, Y_r) < \CS (W, Y_1)$, contradicting our initial assumption that $\CS (W, Y_1) \leq \CS (W, Y_i)$ for all $i \in \{ 1, \dots, n \}$. Thus $r < s$. Let
\[ Z : = \ray (\lm_r y_r + \lm_s y_s) \in [Y_r, Y_s]. \]
Formula \eqref{eq:11.3} for this ray $Z$ tells us that the sum on the left in ($\ast$) equals $\CS (W, Z)$. Thus
\begin{equation*}\label{eq:str.2} \CS (W, Z) \leq \CS (W, Y) < \CS (W, Y_1). 
\end{equation*}
It follows that $\CS (W, Z)$ is smaller than both $\CS (W, Y_r)$ and $\CS (W, Y_s)$. By our analysis of the minimal value of $\CS (W, \udscr )$ on closed intervals (Theorem~\ref{thm:10.7}),  we conclude that $\CS (W, \udscr )$ is not monotone on $[Y_r, Y_s]$. So, the minimum of $\CS (W, \udscr)$ on $[Y_r, Y_s]$ is attained at the $W$-median $M_W (Y_r, Y_s)$ and only at this ray.
It is now evident that the minimum of $\CS (W, \udscr)$ on~ $C$ exists and is attained at one of the finitely many $W$-medians $M_W (Y_i, Y_j)$ with $\CS (W, \udscr)$ not monotone on $[Y_i, Y_j]$. \end{proof}

Given a finitely generated convex set $C$ in $\Ray (V)$, a sequence $(Y_1, \dots, Y_n)$ of generators of $C$, and  a ray $W$ on $V$, we define
 \begin{equation}\label{eq:11.4} \mu_W (Y_1, \dots, Y_n) : = \mu_W (C) : = \min\limits_{Z \in C} \CS (W, Z). \end{equation}

\begin{thm}\label{thm:11.2} In this setting assume  for a fixed $W$ that
\[ \mu_W (C) < \min\limits_{1 \leq i \leq n} \CS (W, Y_i), \]
 that   $Y \in C$ is a ray with $\CS (W, Y) = \mu_W (C)$,  and   that a presentation $Y = \ray (\lm_1 y_1 + \dots + \lm_n y_n)$ is given with $y_i \in e Y_i$, $\lm_i \in eR$.
Then for every dominant term $\gamma_{rs} \lm_r \lm_s$ in the sum ~\eqref{eq:11.2}  above we have $r < s$. The ray
\[ Z_{rs} : = \ray (\lm_r y_r + \lm_s y_s) \in [Y_r, Y_s] \]
is the $W$-median of $[Y_r, Y_s]$, and
 \begin{equation}\label{eq:11.5} \mu_W (C) = \mu_W (Y_r, Y_s) = \CS (W, Z_{rs}) = \sqrt{\frac{\CS (W, Y_r) \CS (W, Y_s)}{\CS (Y_r, Y_s)}}. \end{equation}
Moreover, $Z_{rs}$ is the unique ray $Z$ in $[Y_r, Y_s]$ for which  $\CS (W, Z) = \mu_W (C)$.
\end{thm}
\begin{proof} By the arguments in the proof of Theorem \ref{thm:11.1}, for Case B, we have
\[ \CS (W, Z_{rs}) \leq \CS (W, Y) = \mu_W (C). \]
Trivially
\[ \mu_W (C) \leq \mu_W (Y_r, Y_s) \leq \CS (W, Z_{rs}). \]
We conclude that equality holds here everywhere. Since $\mu_W (Y_r, Y_s) = \mu_W (C)$ is smaller than both $\CS (W, Y_r)$ and $\CS (W, Y_s)$, the function $\CS (W, \udscr)$ on $[Y_r, Y_s]$ is certainly not monotone. We conclude by Theorem~\ref{thm:10.7}.d, that the most right equality in \eqref{eq:11.5} holds, as well as  the last assertion in the theorem.
\end{proof}

\begin{cor}\label{cor:11.3} Assume that $Y_1, \dots, Y_n, W, W'$ are rays in $V$ with $\CS (W, Y_i) = \CS (W', Y_i)$ for $1 \leq i \leq n$. Then
\[ \mu_W (Y_1, \dots, Y_n) = \mu_{W'} (Y_1, \dots, Y_n). \]
\end{cor}

\begin{proof}
 Let $C = \mbox{conv} (Y_1, \dots, Y_n)$. We shall infer from \S\ref{sec:7}, \S \ref{sec:8}, and Theorem~\ref{thm:11.1} that the value $\mu_W (C) = \mu_W (Y_1, \dots, Y_n)$ is uniquely determined by the quantities $\CS (W, Y_i)$, $1 \leq i \leq n$. This is trivial for $n = 1$, while for $n = 2$,
\[ \mu_W (C) = \min (\CS (W, Y_1), \CS (W, Y_2)) \]
except in the case that the profile of $\CS (W, -)$ on $[Y_1, Y_2]$ is not monotone. This property only depends on the values $\CS (W, Y_1)$ and $\CS (W, Y_2)$ (cf. Table \ref{table:7.3}). Then
\[ \mu_W (C) = \sqrt{\frac{\CS (W, Y_1) \CS (W, Y_2)}{\CS (Y_1, Y_2)}}. \]

When $n \geq 3$ it follows from Theorem \ref{thm:11.1} (and more explicitly from  Theorem~\ref{thm:11.2}) that $\mu_W (C)$ is determined by the values $\CS (W, Y_i)$, $1 \leq i \leq n$, and those values $\mu_W (Y_r, Y_s)$, $1 \leq r < s \leq n$, which are smaller than $\CS (W, Y_r)$ and $\CS (W, Y_s)$. Thus in all cases $\mu_W (C)$ remains unchanged if we replace $W$ by $W'$. \end{proof}

Formula \eqref{eq:11.3} has been the main new ingredient for proving Theorems~\ref{thm:11.1} and \ref{thm:11.2}. We quote another (immediate) consequence of this formula.

\begin{thm}\label{thm:11.4} Assume the $eR$-module $eV$ is free with base $y_1, \dots, y_n$.
Let $y \in eV$, $Y = \ray(y)$, and $Y_i = \ray(y_i)$.
Then
 \begin{equation}\label{eq:11.6} \CS (W, Y) \leq  \frac{q_{\QL} (y)}{q (y)} \cdot \max\limits_{1 \leq i \leq n} \CS (W, Y_i). \end{equation}
\end{thm}
\noi
This theorem is a sharpening of \cite[Theorem~7.9.a]{QF2} in the free case. Conversely it is immediate to deduce the quoted result in \cite{QF2} from  \eqref{eq:11.6} by pulling back the quadratic pair $(eq,eb)$ to a free module.
\pSkip

We now study the minimal values of $\CS (W, - )$ on the convex hulls of subsets of $\{ Y_1, \dots, Y_n \}$.

\begin{defn}\label{def:11.5} Given a finite set $S = \{ Y_1, \dots, Y_n\}$ of rays in $V$ and a ray $W$ in $V$, we define the  subset
\[ \overrightarrow{\mu}_W (S) : = \overrightarrow{\mu}_W (Y_1, \dots, Y_n), \]
of  $e R$ as follows: \begin{equation}\label{eq:11.7} \overrightarrow{\mu}_W (S) : = \{ \mu_W (T) \ds  \vert T \subset S, T \ne \emptyset \}, \end{equation}
i.e., $\overrightarrow{\mu}_W (S)$ is the set of minimal values of $\CS (W, - )$ on the convex hulls of all nonempty subsets of $S$. We call $\overrightarrow{\mu}_W (S)$ the $\CS$-\textbf{spectrum} of $W$ on the set of rays~$S$.
\end{defn}

We list the finite poset $\overrightarrow{\mu}_W (S)$ as a sequence
 \begin{equation}\label{eq:11.8} \mu^0_W (S) < \mu^1_W (S) < \dots < \mu^m_W (S) \end{equation}
in $eR$. Here $\mu^0_W (S)$ and $\mu^m_W (S)$ are the minimum and the maximum, respectively, of $\CS (W, - )$ on the convex hull $C = \conv (S)$ of $S$. Notice that the other values $\mu^{i}_W (S)$ will often depend on the set of generators $S$ of the convex set $C$ instead of $C$ alone.

%

Aa a consequence of Theorem \ref{thm:11.1} we have the following fact.

\begin{schol}
\label{schl:11.6} Let $S = \{ Y_1, \dots, Y_n \}$. The elements of $\overrightarrow{\mu}_W (S)$ are the values $\CS (W, Y_i)$, $1 \leq i \leq r$, and $\CS (W, M_W (Y_r, Y_s))$ where $(Y_r, Y_s)$ runs through all pairs in $S$ such that $\CS (W, -)$ is not monotone on $[Y_r, Y_s]$. \end{schol}

\begin{prop}
\label{prop:11.7} Assume that $P$ and $Q$ are subsets of $S$, such that all intervals $[Y, Z]$ with $Y \in P$, $Z \in Q$ have a monotone $W$-profile. Then
 \begin{equation}\label{eq:11.9} \overrightarrow{\mu}_W (P \cup Q) = \overrightarrow{\mu}_W (P) \cup \overrightarrow{\mu}_W (Q). \end{equation}
In particular, this holds \textbf{for every} $W \in \Ray (V)$, if the quadratic form $eq$ is quasilinear on these intervals $[Y, Z]$.
\end{prop}
\begin{proof}
This is evident from the preceding description of CS-spectra. \end{proof}

\section{The glens and the glen locus of a finite set of rays}\label{sec:12}

As previously, we assume  that the ghost ideal $eR$ of the supertropical semiring $R$ is a semifield and that the quadratic form $q$ on $V$ is anisotropic.

\begin{defn}\label{def:12.1} The \textbf{glen} of a finite sequence of rays $Y_1, \dots, Y_n$ in $V$ \textbf{at a ray} $W$ in $V$ is the set of all $Z \in \conv (Y_1, \dots, Y_n)$ such that
\[ \CS (W, Z) < \min\limits_{1 \leq i \leq n} \CS (W, Y_i). \]
We denote this set by $\Glen_W (Y_1, \dots, Y_n)$, and  call it the $W$-\textbf{glen of} $(Y_1, \dots, Y_n)$, for short.
\end{defn}

For notational reasons we do not exclude the case $n = 1$. Then, of course, all $W$-glens are empty.

\begin{prop}\label{prop:12.2}
  $\Glen_W(Y_1,\dots, Y_n))$ is a convex subset of $\Ray(V)$ (perhaps empty).
\end{prop}
\begin{proof} Given three rays $Z_1, Z_2, Z$ in conv$(Y_1, \dots, Y_n)$ with $Z \in [Z_1, Z_2]$ and $\CS (W, Z_1) < \CS (W, Y_i)$, $\CS (W, Z_2) < \CS (W, Y_i)$ for all $i \in [1, n]$, we infer from Theorem  ~\ref{thm:11.4} that $\CS (W, Z) < \CS (W, Y_i)$ for all $i \in [1, n]$. \end{proof}

For $n = 2$ and $Y_1 \ne Y_2$ the rays $Y_1$ and $Y_2$ are uniquely determined, up to permutation, by the closed interval $[Y_1, Y_2]$, as we know. Thus we are entitled to define the $W$-\textbf{glen of} $[Y_1, Y_2]$ as
\[ \Glen_W [Y_1, Y_2] ; = \Glen_W (Y_1, Y_2). \]
From the analysis of the function $CS (W, - )$ on closed intervals in \S \ref{sec:6} we infer the following statement, which justifies the use of the name ``glen'' at least for $n = 2$. See also Figure 4 below.

\begin{schol}
\label{schl:12.2} $\Glen_W [Y_1, Y_2]$ is not empty iff the $W$-profile on $[Y_1, Y_2]$ is not monotonic, and thus  is of type $B$ or $B'$ or $\partial B$ (cf. \S \ref{sec:7}). \end{schol}

Relying on \S \ref{sec:6}, we give an explicit description of the $W$-glen of $[Y_1, Y_2]$ in the case of type $B$ or $\partial B$, i.e., when $\CS (W, - )$ is not monotonic on $[Y_1, Y_2]$ and $\CS (W, Y_1) \leq (W, Y_2)$. Choosing vectors $\veps_1, \veps_2, \veps_3 \in V$ such that $W = \ray (\veps_1)$, $Y_1 = \ray (\veps_2)$, $Y_2 = \ray (\veps_3)$, we have the following illustration  of the  function $f (\lm) = \CS (\veps_1, \veps_2 + \lm \veps_3) = f_1 (\lm) + f_2 (\lm)$, using the notations from \S \ref{sec:6}.

 \begin{figure}[h]\label{fig:4}
\resizebox{0.5\textwidth}{!}{

\begin{tikzpicture}[]
\draw[thick] (-4,0) --(0,0) ;

\draw[thick] (0,0) node[below]{$\frac{\al_2}{\al_{23}}$} --(15,0)node[below]{$\lm$} ;
\draw[blue,thick] (0,0) -- (0,10);
\draw[blue, thick] (0,8) node[left]{$CS(\veps_1,\veps_3)$}--(13,8);
\draw[blue, thick] (0,6) --(13,6)node[right]{${CS(\veps_1,\veps_2)}$};

\draw[blue, thick](9,0)node[below]{$\zeta$}--(9,10);
\draw[blue, thick](13,0)node[below]{$\frac{\al_{23}}{\al_3}$}--(13,10);

\draw[blue, thick](4.5,0)node[below]{$\xi$}--(4.5,10);
\draw[help lines] (0,0) grid (13,10);
\draw[red, line width=0.5mm] (0,2) node[left]{}--(13,8);
\draw[red, line width=0.5mm] (13,8) --(15,8) node[right]{$\bf f_2$};

\draw[red, line width=0.5mm] (-3,6) node[left]{$f_1$}--(0,6);
\draw[red, line width=0.5mm] (0,6) node[left]{}--(4.5,4);

\end{tikzpicture}
}
\caption{  $\xi=\frac{\al_{12}}{\al_{13}}$. 
}
\end{figure}

The $W$-glen of $[Y_1, Y_2]$ is contained in the open interval $] Y_{12}, Y_{21} [$ , where $Y_{12}$ is the characteristic ray of $[Y_1, Y_2]$ near $Y_1$ and $Y_{21}$ is the characteristic ray near $Y_2$. It starts at the argument $\lm = \frac{\al_2}{\al_{23}}$ corresponding to $Y_{12}$ and ends at the argument $\zeta \in \,  ] \frac{\al_2}{\al_{23}}, \frac{\al_{23}}{\al_{3}} [$ with $f_2 (\zeta) = \CS (\veps_1, \veps_2)$. In the interval $[ \frac{\al_2}{\al_{23}}, \frac{\al_{23}}{\al_{3}}]$ the function $f_2$ reads
\[ f_2 (\lm) = \frac{\lm^2 \al^2_{13}}{\al_1 \cdot \lm \al_{23}} = \frac{\lm \al^2_{13}}{\al_1 \cdot \al_{23}} \]
(cf. \eqref{eq:6.3}), since here the term $\lm \al_{23}$ is dominant in the formula \eqref{eq:6.8} for $q (\veps_2 + \lm \veps_3)$. Thus we have to solve
\[ \frac{\zeta \al_{13}^2}{\al_1 \al_{23}} = \frac{\al_{12}^2}{\al_1 \al_2}, \]
and obtain
 \begin{equation}\label{eq:12.1} \zeta = \frac{\al_{12}^2}{\al_{13}^2} \cdot  \frac{\al_{23}}{\al_2}.  \end{equation}
In the subcase $\CS (\veps_1, \veps_2) = \CS (\veps_1, \veps_2)$, i.e., $\frac{\al_{12}^2}{\al_1 \al_{2}} = \frac{\al_{13}^2}{\al_1 \al_3}$, we obtain
 \begin{equation}\label{eq:12.2} \zeta = \frac{\al_{23}}{\al_3}. \end{equation}

Introducing the ray
\[ Z_{21} : = \ray (\veps_2 + \zeta \veps_3) = \ray (\al_{13}^2 \al_2 \veps_2 + \al_{12}^2 \al_{23} \veps_3), \]
we summarize our study of the $W$-glen of $[Y_1, Y_2]$ as follows.

\begin{thm}\label{thm:12.3}
 Assume that $\CS (W, Y_1) \leq \CS (W, Y_2)$ and $\Glen_W [Y_1, Y_2] \ne \emptyset$. Then
 \begin{enumerate}
 \ealph
   \item $\Glen_W [Y_1, Y_2] =\, ] Y_{12}, Z_{21} [\, \subset \, ] Y_{12}, Y_{21} [$,
   \item $\Glen_W [Y_1, Y_2] = \, ] Y_{12}, Y_{21} [$ iff $\CS (W, Y_1) = \CS (W, Y_2)$.
\end{enumerate}\end{thm}
\begin{rem}
\label{rem:12.4} Note that

\[ \sqrt{ \frac{\al_2}{\al_{23}} \cdot \frac{\al^2_{12}}{\al_{13}^2} \cdot  \frac{\al_{23}}{\al_2}} = \frac{\al_{12}}{\al_{13}} = \xi. \]
Thus the median $M_W (Y_1, Y_2)$ may be viewed as a kind of geometric mean of the rays $Y_{12}$ and ~ $Z_{21}$.
\end{rem}

We now look at the set of rays $W$ where a nonempty glen of $(Y_1, \dots, Y_n)$ occurs.

\begin{defn}\label{def:12.5} The \textbf{glen-locus} of $(Y_1, \dots, Y_n)$ is the set
\[ \begin{array}{ll}
 \Loc_{\glen} (Y_1, \dots, Y_n) & : =  \{ W \in \Ray (V) \ds \vert \Glen_W (Y_1, \dots, Y_n) \ne \emptyset \}\\[1mm]
&\; = \big\{ W \in \Ray (V) \ds \vert \mu_W (Y_1, \dots, Y_n) < \min\limits_{1 \leq i \leq n} \CS (W, Y_i) \big \}.
\end{array} \]
\end{defn}
For $n = 2$, $Y_1 \ne Y_2$, we define
\[ \Loc_{\glen} [Y_1, Y_2] : = \Loc_{\glen} (Y_1, Y_2), \]
which makes sense, since the rays $Y_1, Y_2$ are uniquely determined, up to permutation,  by the interval $[Y_1, Y_2]$. Theorem ~\ref{thm:11.2} translates into the following statement, where we use the definition of loci of basic and composed profile types from \S \ref{sec:8} (Definitions~\ref{def:8.5} and \ref{def:8.7}).
\begin{schol}
\label{schl:12.6} $\Loc_{\Glen} [Y_1, Y_2]$ is the disjoint union of the basic loci $\Loc_B [\overrightarrow{Y_1, Y_2}]$, $\Loc_{B'} [\overrightarrow{Y_1, Y_2}]$, and $\Loc_{\partial B} [Y_1, Y_2]$, which are disjoint convex subsets of $\Ray (V)$.It is also the union of the composed loci $\Loc_{\overline{B}} [\overrightarrow{Y_1, Y_2}]$, $\Loc_{\overline{B'}} [\overrightarrow{Y_1, Y_2}]$, which are again convex, and have the intersection $\Loc_{\partial B} [Y_1, Y_2]$ (cf. Theorems~\ref{thm:8.6} and \ref{thm:8.8}). \end{schol}

Of course it may happen that $\Loc_{\Glen} [Y_1, Y_2]$ is empty. In particular this occurs if $[Y_1, Y_2]$ is $\nu$-quasilinear.

Given a set $\{ Y_1, \dots, Y_n \}$ of rays in $V$ with $n > 2$, we have the important fact in consequence of Theorem~\ref{thm:11.1} (cf. Scholium~\ref{schl:11.6}), that $\Loc_{\glen} (Y_1, \dots, Y_n)$ is contained in the union of all sets $\Loc_{\glen} (Y_r, Y_s)$ with $1 \leq r < s \leq n$, such that $[Y_r, Y_s]$ is $\nu$-excessive \cite[Definition ~7.3]{QF2}, equivalently, $\Loc_{\glen} [Y_r, Y_s] \ne \emptyset$. If there are $u = u (Y_1, \dots, Y_n)$ such pairs~ $(r, s)$, then $\Loc_{\glen} (Y_1, \dots, Y_n)$ is a union of at most $2 u$ convex subsets of $\Ray (V)$.

\section{Explicit description of the set of minima of a CS-function on a finitely generated convex set in the ray space}\label{sec:13}

\emph{In the whole section  $R$ is a supertropical semiring,  $eR$ is a nontrivial semifield, and  $(q,b)$ is a quadratic pair on an $R$-module $V$ with $q$ anisotropic}. Given a finite subset $S \subset \Ray(V)$ and a fixed ray $W$ in $V$, we explore the set of minima of $\CS(W,-)$ on the convex hull $C$ of~ $S$ in $\Ray(V)$, denoted by $\Min\CS(W,C)$.
We already know that $\Min\CS(W,C)$ is nonempty.  Our first goal is to prove that $\Min\CS(W,C)$ is again a finitely generated convex  subset of $\Ray(V)$, and to determine a set of generators of $\Min\CS(W,C)$ starting from $S$.

\begin{thm}\label{thm:13.1} $ $
   \begin{enumerate}
 \ealph
 \item $\Min\CS(W,C)$  is convex.

\item If $X \in \Min\CS(W,C)$  and $Y \in C \setminus \Min\CS(W,C)$, then  $[X,Y] \cap \Min\CS(W,C) = [X, M_W(X,Y)]$.

\item If $\CS(W,-)$ is constant on $S$, then $\CS(W,-)$ is constant on $C$.

\item Assuming that $\CS(W,-)$ is not constant on $S$. Let $S^*$ denote the set of all medians $M_W(X,Y)$
with $X,Y \in S$ and $\CS(W,M_W(X,Y) ) = \mu_W(S)$, \footnote{Recall that $\mu_W(S) =\mu_W(C)$ denotes the minimal value of $\CS(W,-)$ on $S$, and hence on $C$.} which may be empty.
Let
$$
\begin{array}{ll}
P := & (S \cup S^*) \cap \Min\CS(W,C), \\[1mm]
Q := & S \setminus P = (S \cup S^*) \setminus P.
\end{array}$$
 Then $\Min\CS(W,C)$  is the convex hull of the finite set
$P \cup M_W(P,Q)$, where
$$ M_W(P,Q) := \{ M_W(X,Y) \ds | X \in P, Y \in Q\}. $$
\end{enumerate}
\end{thm}
\begin{proof} (a): If $X,Y$ are rays in $\Min\CS(W,C)$, then $\CS(W,X) = \CS(W,Y)$, from which we conclude that $\CS(W,-)$ is constant on $[X,Y]$, since no $W$-glen is possible on $[X,Y]$. Thus $[X,Y] \subset \Min\CS(W,C)$.
  \pSkip
  (b): $\CS(W,Y)$ attains its minimal value on $[X,Y] $ in $X$. Thus it is clear from our analysis of $\CS(W,-)$ on
  $[\overrightarrow{X,Y}]$ in \S\ref{sec:6},  that the set $Z$ of rays  in $[X,Y]$ with $\CS(W,Z) = \CS(W,X)$ is $[X, M_W(X,Y)]$, cf. \eqref{eq:6.12} and \eqref{eq:6.13}.
  \pSkip
  (c): Assume that $\CS(W,-)$ is constant on $S$, then $\CS(W,-)$ has no glens at all, and we conclude as in (a) that $\CS(W,-)$ is constant on $C$.
  \pSkip
  (d): It follows from (b) that $M_W(P,Q)$ is contained  in $\Min\CS(W,C)$ and then by (a) that
  $$\conv(P \cup  M_W(P,Q)) \subset \Min\CS(W,C).$$
  We now verify that any given ray $Y \in \Min\CS(W,C)$ is contained in the convex hull of $P \cup M_W(P,Q)$.
  Let $S = \{Y_1,\dots,Y_n\}$.
  We choose a minimal subset $\{Y_k \ds| k \in K \}$ of $S$, $K \subset \{1, \dots, n\}$, such that
  \begin{equation}\label{eq:13.1}
    Y \in \conv(P \cup \{Y_k \ds| k \in K \}).
  \end{equation}
  If $K= \emptyset$, then $Y\in\conv(P)$, and we are done.

   In the case that $K \neq \emptyset$ we choose a minimal set of rays $\{Z_j \ds | j \in J\}$  in $P$ such that
   \begin{equation}\label{eq:13.2}
     Y = \conv(\{Z_j \ds | j \in J\} \cup \{Y_k \ds | k \in K\}),
   \end{equation}
   so that $Y \in \conv(A \cup B)$ where
   $$ A = \conv(\{Z_j \ds | j \in J\}), \qquad B = \conv( \{Y_k \ds | k \in K\}).$$
  Then $A \subset P \subset \Min\CS(W,C)$ while $B$ is disjoint from $P$ due to the minimality of the set
     $\{Y_k \ds | k \in K\}$ in \eqref{eq:13.1}.  Since $Y \in \Min\CS(W,C)$, we conclude by assertion (b) that $Y \in M_W(Z,T)$ for some rays $Z \in A$, $T \in B$. Choosing vectors $z_j \in e Z_j$, $y_k \in e Y_k$, $w \in eW$ we have $Y = \ray (y)$, with
     \begin{equation}\label{eq:13.3}
        y = m_w \bigg( \sum_{j\in J} \mu_j z_j, \sum_{k\in K} \lm_k y_k\bigg)
    \end{equation}
    and nonzero coefficients $\mu_j, \lm_k \in eR$.
    Thus (cf.  \eqref{eq:10.1} and \eqref{eq:10.2})
\begin{align*}
  y  & = b\bigg(w,\sum_{k\in K} \lm_k y_k\bigg) \sum_{j\in J} \mu_j z_j + b\bigg(w, \sum_{j\in J} \mu_j z_j\bigg) \sum_{k\in K} \lm_k y_k \\ & = \sum_{j\in J, k\in K } \mu_j \lm_k \big( b(w,y_k) z_j + b(w,z_j) y_k\big),
\end{align*}
which proves that
\begin{equation}\label{eq:13.4}
  Y \in \conv \big(\{ M_W(Z_j, Y_k) \ds | j \in J, k \in K \} \big) \subset \conv(M_W(P,Q)).
\end{equation}
\end{proof}

It is to be expected from Theorem \ref{thm:13.1} that usually many more rays are needed to generate the convex set
$\Min\CS(W,C)$ than to generate $C$, But now we exhibit cases, where  $\Min\CS(W,C) = \Min\CS(W,S)$ can be generated by very few rays.
\begin{prop}\label{prop:13.2}
  If $\mu_W(S) = 0$, then $\Min\CS(W,C)$ is the convex hull of
  $$ W^\perp \cap S = \{Z \in S \ds |  \CS(W,Z) = 0 \}.   $$
\end{prop}
\begin{proof}
  By the results in \S\ref{sec:10} there are no rays  $X,Y$ in $V$ with $\CS(W,X) > 0 $, $\CS(W,Y) > 0 $,
  $\CS(W,M_W(X,Y)) = 0 $ (cf. e.g. \eqref{eq:10.9}). Thus $S^* = 0$, and we conclude from Theorem \eqref{thm:13.1} that
  $\Min\CS(W,C) = \conv(W^\perp \cap S)$.
\end{proof}
\begin{examp}\label{exmp:13.3}
Assume that $S = \{X_1,\dots,X_n\}$, $n \geq 2$, is a finite set of rays in $V$ such that there exists a ray $Z$ with
$0 < \CS(W,Z) <\CS(W,X_i)$ for every $i \in \{1,\dots, n\}$ and $M_W(X_i,X_j)=Z$ for $1 \leq i < j \leq n$.
Then $S^* = \{ Z \}$, and $M_W(X_i, X_j) = Z$ implies that $\CS(W,-)$ is strictly increasing on $[\overrightarrow{Z,X_i}]$ (and  $[\overrightarrow{Z,X_j}]$), cf. \S\ref{sec:6}, whence $M_W(Z,X_i) = Z$ for every $i$, Thus $\CS(W,S) = \{ Z \}$.
\end{examp}

\begin{defn}\label{def:13.4}
We call a set $S $  as described in Example \ref{exmp:13.3} a \textbf{median cluster for~ $W$}, or $W$-median cluster, \textbf{with apex $Z$}.
\end{defn}

\begin{examp}\label{exmp:13.5}
Assume that $S = P_1 \cup P_2$ is a disjoint union of two $W$-median clusters $P_1$ and $ P_2$  with  apices $Z_1$ and $Z_2$. Assume further that all pairs $X,Y$ with $X \in P_1$, $Y \in P_2$ have the same median $M_W(X,Y) = Z_{12}$ and that
$$ \CS(W,Z_{1})= \CS(W,Z_{2}) = \CS(W,Z_{12}).$$
Then we conclude from Theorem \ref{thm:13.1} that
\begin{equation}\label{eq:13.5}
  \Min\CS(W,S) = \conv(Z_1, Z_2, Z_{12}).
\end{equation}
Indeed, in the notation there $ P = S^* = \{ Z_1, Z_2\}$, $Q= S = P_1 \cup P_2,$
and so $M_W(P,Q) = \{ Z_1, Z_2, Z_{12} \}$.
In the case  $Z_1 = Z_2$ we obtain
\begin{equation}\label{eq:13.6}
  \Min\CS(W,S) = [Z,  Z_{12}]
\end{equation}
with $Z:= Z_1 = Z_2$.
\end{examp}

Given a ray $Z$ in $V$, we now focus on the set of all  $W$-median clusters in $\Ray(V)$ with apex~ $Z$.
We assume that $\CS(W,Z)>0$, since otherwise it is clear from Proposition \ref{prop:13.2}, that there are no median clusters with apex $Z$. The next lemma, a simplification of an argument in the proof of Theorem \ref{thm:13.1}.d (\eqref{eq:13.1}--\eqref{eq:13.4}),  will be of help.
\begin{lem}[$M_W$-Convexity Lemma]\label{lem:13.6}
Let $Z \in \Ray(V)$. \footnote{Here it is not necessary to assume that $\CS(Z,W)>0$.}
  Assume that $P = \{ Y_j \ds | j \in K\}$ and $Q = \{ Y_k \ds | k \in J\}$ are disjoint sets of rays with
  $M_W(Y_j, Y_k) = Z$ for any $Y_j \in P$ and $Y_k \in Q$. Then also $M_W(Y,T) = Z$ for any $Y \in \conv(P)$ and  $T \in \conv(Q)$.
\end{lem}
\begin{proof}
  Given $Y \in P $, $T \in Q$ we choose vectors $y_j \in e Y_j$, $y_k \in e Y_k$, $z \in eZ$, $t \in eT$.
  Then $Y = \ray(y)$, $T = \ray(t)$ with $y = \sum\limits_{j \in J} \lm_j y_j$, not all $\lm_j = 0$, and
  $y = \sum\limits_{k \in K} \mu_k y_k$, not all $\mu_k = 0$. Since $M_W(Y_j,Y_k) = Z$ for $j \in J $, $k \in K$,  we have
  $$ b(w,y_k) y_j + b(w,y_j) y_k = m_w(y_j, y_k) = \al_{jk }z,$$
  for these indices $j,k$, with   $\al_{jk}\neq 0$. Thus
  \begin{align*}
    b(w,t)y + b(w,y)t &  = \sum_{k \in K } \mu_{k} b(w,y_k) \sum_{j\in J } \lm_j y_j +
    \sum_{j\in J } \lm_j b(w,y_j) \sum_{k\in K } \mu_k y_k \\
    & = \sum_{j\in J, k\in K} \lm_j \mu_k [b(w,y_k) y_j + b(w,y_j)y_k]
\\ &
=\bigg( \sum_{j\in J, k\in K} \al_{jk} \lm_j \mu_k \bigg)z.
  \end{align*}
  Since $\sum\limits_{j\in J, k\in K} \al_{jk} \lm_j \mu_k \neq 0$, this proves that $M_W(Y,T) = Z$.
\end{proof}
Given a ray $Z$ in $V$ with $\CS(W,Z) >0$, we introduce the ray set
\begin{equation}\label{eq:13.7}
  \Zup := \{ X \in \Ray(V) \ds | \CS(W,X) > \CS(W,Z) \} .
\end{equation}
This set contains every $W$-median cluster having  apex $Z$. Note that typically the set $\Zup$ is not convex.
\begin{defn}\label{def:13.6}
Let $P \subset Z$, $P \neq \emptyset$. The \textbf{$Z$-polar of $P$ for $W$} (or $W$-$Z$-polar of $P$) is the set
\begin{equation}\label{eq:13.8}
  \chP = P^\vee := \{ Y \in \Zup \ds |  \exists X \in P : M_W(X,Y) = Z \} .
\end{equation}

\end{defn}
Note that
\begin{equation}\label{eq:13.9}
P_1 \subset P_2 \subset \Zup \dss \Rightarrow \chP_2 \subset \chP_1,
\end{equation}
and, that
\begin{equation}\label{eq:13.10}
\bigg( \bigcup_{\lm \in \Lm } P_\lm \bigg)^\vee =   \bigcap_{\lm \in \Lm } \chP_\lm
\end{equation}
for any family $(P_\lm \ds | \lm \in \Lm )$ of subsets $P_\lm$ of $\Zup$.

\begin{remarks}\label{rem:13.8}
  Let $P \subset \Zup$, $P \neq \emptyset$.
  \begin{enumerate}\ealph
    \item Then $P$ and $\chP$ are disjoint, since $M_W(X,X) = X \neq Z$ for every $X \in P$.
    \item If $\chP \neq \emptyset$, then $P \subset P^{\vee \vee}$, This implies in the usual way that
  \begin{equation}\label{eq:13.11}
    P^{\vee\vee\vee} = P^{\vee}.
  \end{equation}
  \end{enumerate}
\end{remarks}
We define for $P \subset \Zup$ the set
  \begin{equation}\label{eq:13.12}
    \conv_0(P) := \conv(P) \cap \Zup.
  \end{equation}

\begin{thm}\label{thm:13.9}
Let $P \subset \Zup$ and $P \neq \emptyset$, then
$$\chP =  \conv_0(\chP) = \conv_0(P)^\vee.$$
\end{thm}
\begin{proof}
  We have $P \cap \chP = \emptyset$ and $M_W(X,Y) = Z$ for $X \in P $, $Y \in \chP$. By the Median Convexity Lemma  \ref{lem:13.6}, this implies $M_W(X',Y') = Z$ for $X' \in \conv(P)$ and $Y'\in \conv(\chP)$. Thus $\conv_0(\chP) \subset \conv_0(P)^\vee $. We further infer from $P \subset \conv_0(P)$ that $\conv_0(P)^\vee  \subset \chP$, and so
  $\conv_0(\chP) \subset \chP$. Since trivially
  $\chP \subset \conv_0(\chP)$, this proves that $\chP =  \conv_0(\chP) = \conv_0(P)^\vee.$
\end{proof}

We now  employ the partial ordering $\leq_Z$ on $\Ray(V)$, given by
$$ Y' \leq_Z Y \dss \Leftrightarrow [Z,Y'] \subset [Z,Y],$$
the basics of which can be found in \cite[\S8]{VR1}. This ordering extends the total ordering on the oriented intervals $[\overrightarrow{Z,Y}]$ used in the previous sections.
\begin{thm}\label{thm:13.10}
For any nonempty subset $P$ of $\Zup$ the $Z$-polar $\chP$ is compatible with $\leq_Z$ in the following sense.
If $Y,Y' \in \Zup$ and $Y \leq_Z Y'$, then
\begin{equation}\label{eq:13.13}
Y \in \chP \dss \Leftrightarrow Y' \in \chP.
\end{equation}
\end{thm}

\begin{proof}
  This follows from the fact that for any $X \in P$ the CS-function $\CS(W,-)$ is not monotonic on $[X,Y]$ iff it is not monotonic on $[X,Y']$, and then $\CS(W,- )$ attains its unique minimum at
  $M_W(X,Y) = M_W(X,Y')$, as is clear from \S\ref{sec:6} and \S\ref{sec:7}, cf. Figures 1--3 in \S\ref{sec:6}.
\end{proof}

We describe a procedure to build up clusters with apex $Z$, basing on some more terminology.
For any ray $X \in \Zup$, we write $\chX = X^\vee = \{ X\} ^\vee  $ for short.

\begin{defn}\label{def:13.10}
We say that $X$ is $Z$-polar, if $\chX \neq \emptyset$, and so $X$ is in the $Z$-polar of the set $\chX$. More explicitly, $X$ is $Z$-polar, if $M_W(X,Y) = Z$ for some $Y \in \Zup$.
\end{defn}

Note that for any set $P \subset \Zup$ we have
\begin{equation}\label{eq:13.14}
  \chP = \bigcap_{X \in P} \chX.
\end{equation}
If $\Zup$ does not contain  $Z$-polar sets, then, of course, there do not exist median clusters with apex $Z$. Otherwise we choose $X_1,X_2 \in \Yup$ with $M_W(X_1,X_2) = Z$.  If $\{ X_1, X_2 \}^\vee = X_1^\vee \cap X_2^\vee \neq \emptyset$, we choose a ray $X_3 \in \Zup$ with  $X_3 \in \{ X_1, X_2 \}^\vee$.
Proceeding in this way we obtain a sequence of rays $X_1, \dots, X_r$ in $\Yup$ with $r \geq 2$ and
\begin{equation}\label{eq:13.15}
  X_{i+1} \in \{X_1, \dots, X_i \}^\vee \qquad \text{for $1 \leq i < r$.}
\end{equation}

There are two cases.
\begin{description}
  \item[\underline{Case A}]  We reach a set $S = \{X_1, \dots, X_r \}$ with $\{X_1, \dots, X_r \}^\vee = \chX_1 \cap \cdots \cap \chX_r = \emptyset$. Then $S$ is a maximal median cluster with apex $Z$.

  \item[\underline{Case B}] We obtain infinite sets $S \subset \Zup$, such that every finite subset
  $T \subset S$, $|T| \geq 2$, is a $W$-median cluster with apex $Z$. We call such set $S$ a \textbf{generalized $W$-median cluster} with apex $Z$ (or generalized $W$-$Z$-median cluster). More specifically, using mild set theory,  we obtain by a transfinite  induction procedure  a sequence of rays  $\{ X_i \ds | 1 \leq i \leq \lm \} $ with ordinal $\lm \geq \om$ which is a \textbf{maximal generalized $W$-$Z$-median cluster}.
\end{description}

\section{The equal polar relation}\label{sec:14}

Let $W$ and $Z$ be any rays in $V$. Given $X_1, X_2 \in \Zup$, cf. \eqref{eq:13.7}, we say that $X_1$ and $X_2$ are \textbf{$W$-$Z$-equivalent} (or $Z$-equivalent for short), and write $X_1 \sim_Z X_2$, if $\chX_1 = \chX_2$.
We call  this equivalence relation on $\Zup$ the \textbf{equal polar relation} for $W$ and $Z$ (or the \textbf{$W$-$Z$-equivalence relation}). For this relation, the equivalence class of a ray $X \in \Zup$  is denoted by
$$ [X] := [X]_Z := [X]_{W,Z}.$$
Note that, if $\chX_1 \neq \emptyset$, then $X_1 \sim_Z X_2$ iff $X_1^{\vee\vee} = X_2^{\vee\vee}$, cf. \eqref{eq:13.11}. We then abbreviate $X^{\vee \vee} = \tlX$.

For most problems concerning $Z$-polars of rays, and in particular
 all problems appearing in ~\S\ref{sec:13},  only the class $[X]_{W,Z}$ matters.
  For example, in a (generalized, maximal) median cluster~ $P$ with apex $Z$ we may replace any $X \in P$ by an $Z$-equivalent ray $X'$, and  have again a (generalized, maximal) median cluster $P'$ with apex $Z$. Therefore, understanding the $W$-$Z$-equivalence is a very basic goal, which we first pose vaguely as follows.

\begin{problem}\label{prob:13.14}
  Describe the pattern of any $Z$-equivalence classes  $[X] \subset \Zup$.
\end{problem}


To approach this problem, so far,  we only know:
\begin{enumerate} \ealph
  \item All rays $X$ with $\chX = \emptyset$ are in one  equivalence class -- the class of non-polar rays (Definition \ref{def:13.10}).
   This is trivial.
   We denote this class by $C_\emptyset$:
    $$C_\emptyset = \{ X \in \Zup \ds | \forall Y \in \Zup: M_W(X,Y) \neq Z\}.$$
    Perhaps it is best to  discard  $C_\emptyset$ from $\Zup$.
  \item $[X]_Z \subset \tlX$. Indeed, if $X_1^{\vee} = X_2^{\vee}$, then $X_1 \in X_1^{\vee\vee} = X_2^{\vee\vee}$.
  \item The relation $\sim_Z$ is compatible with the partial ordering $\leq_Z$ on $\Zup$, i.e., if $X_1$ and $X_2$ are comparable under $\leq_Z$, then $X_1 \sim_Z X_2$, cf. Theorem \ref{thm:13.10}.
\end{enumerate}

We now can point more precisely at the type of questions arising  from  Problem \ref{prob:13.14}.
If~ $\chX \neq \emptyset$, then
$\tlX$ is a convex subset of $\Ray(V)$ contained in $\Zup$, with $[X] \subset \tlX$ by (b). If~ $T \in \tlX$, then $[T] \subset \tlT \subset \tlX$, and so \emph{$\tlX$ is the disjoint union of all classes $[T]$ contained in ~$\tlX$}. Furthermore, since $\tlT$ is convex, also the convex hull $\conv([T])$ of the set $[T]$ is contained in~ $\tlX$. This leads to the next two intriguing  questions.
\begin{enumerate}
  \item[A)] Is $\conv([T])$ also a union of $Z$-equivalence classes?
  \item[B)] When is a class $[T]$  by itself convex?
\end{enumerate}
Due to (c) the whole pattern of classes $[T]$ is compatible with the partial ordering $\leq_Z$.

Concerning question B), so far we have only a partial answer.
\begin{thm}\label{thm:13.13} If $X$ is a $Z$-polar ray in $\Zup$ (i.e., $\chX \neq \emptyset$), then (cf. \eqref{eq:13.13})
$$ [X] = \conv_0\big([X]\big) =  \conv\big([X]\big) \cap \Zup.$$
\end{thm}

\begin{proof}
We need to prove the following. If $X_1, X_2 \in \Zup$, $[X_1, X_2] \in \Zup$, and $\chX_1= \chX_2 \neq \emptyset,$
then $\chX_1 = \chT$ for any $T \in[X_1,X_2]$. We have to verify for any $Y \in \Zup$ that
$$ M_W(X_1,Y) = Z \dss \Leftrightarrow M_W(T,Y) = Z .$$
\pSkip
$(\Rightarrow)$:
If $ M_W(X_1,Y) = Z$, then $M_W(X_2,Y) = Z$, since $\chX_1 = \chX_2$, whence by the $M_W$-Convexity Theorem: $M_W(T,Y) = Z$ for any $T\in [X_1,X_2]$.

\pSkip
$(\Leftarrow)$: Let $S = \{X_1, X_2, Z \} $. The CS-function $\CS(W,-)$ is strictly decreasing on $[\overrightarrow{X_1,Z}]$ and on $[\overrightarrow{X_2,Z}]$, furthermore
$\CS(W,T) > \CS(W,Z)$ and $\CS(W,Y) > \CS(W,Z)$.
 We conclude from this that $\CS(W,-)$ is not monotonic on $[\overrightarrow{T,Y}]$ and has there minimum value $\CS(W,Z)$.
 This implies $M_W(T,Y) = Z$.
\end{proof}
We introduce two more notations around $Z$-equivalence. Recall that
$X \sim_Z T \Iff \tlX = \tlT$ (provided that $\chX  \neq \emptyset$).

\begin{defn}\label{def:13.14}
A \textbf{path} in a class $[X]_Z$ is a sequence of rays $X_0, \dots, X_r $ in $[X]_Z$ where  $[X_{i-1}, X_i] \in \Zup$ for $0 < i \leq r$, $r \geq 1$.
\end{defn}
Note that, in consequence of Theorem \ref{thm:13.13},
\begin{equation}\label{eq:13.16}
  \bigcup_{i = 1} ^r [X_{i-1}, X_i] \subset [X]_Z.
\end{equation}
This gives us an obvious notion of \textbf{path components} of $[X]_Z$. More generally, we may define paths and path components in any subset of $\Zup$.

\begin{defn}\label{def:13.15}
Given rays $X \in \Zup$ and $T\in  [X]_Z$, we define the \textbf{median star} (=$W$-$Z$-median star) $\st_T(X)$
as the set
of all rays $T' $ with $[T, T'] \subset [X]_Z$.
\end{defn}
In other words, $\st_T(X)$ is the union of all intervals $[T,T']$ contained in $[X]_Z$.

\begin{rem}\label{rem:13.16}
If $T',T'' \in \st_T(X)$, then perhaps $[T',T''] \nsubset \Zup$. But, if $[T',T''] \subset \Zup$, then $\conv(T,T', T'' ) \subset [X]_Z$, and so $\conv(T,T', T'' ) \subset \st_T(X)$. Note also that
\begin{equation}\label{eq:13.18}
  \st_T([X]_Z) \subset \tlT \subset \tlX.
  \end{equation}

\end{rem}

Every $Z$-equivalence class $[X]_Z$ is the disjoint union of the path components contained in ~$[X]_Z$. If $A$ and $B$ are such path components, then obviously every interval $[Y_1, Y_2]$ with $Y_1 \in A$, $Y_2 \in B$ has a ``\textbf{deep glen}'' with respect to $Z$, i.e., the median $M_W(Y_1,Y_2)$ is not contained in $\Zup$.
We can refine Problem \ref{prob:13.14} to a description of the pattern of path components of the $W$-$Z$-equivalence classes, which we call the \textbf{refined version of~ Problem~ \ref{prob:13.14}}. This seems to be natural and easier than Problem \ref{prob:13.14} above. Note also that every such path component is the union of all median stars $\st_T(X)$ contained in it.

We have  gained a very rough view to the family of path components of $Z$-equivalence classes as follows. For simplicity, we assume that $eR= \{ 0 \} \cup \tG$ is a nontrivial bipotent  semifield which is \textbf{square-root closed}, i.e., the injective endomorphism $x \mapsto x^2$  is also surjective, and so is an order preserving
automorphism of $eR$. This setup can be reached for any (nontrivial) bipotent semifield by a canonical extension involving only square-roots, cf.~ \cite[\S7]{QF1}.

Assume that $A$ and $B$ are different sets, which are path components of $Z$-equivalence classes different from $ C_\emptyset$. Then for any $T_1\in A$ and $T_2 \in B$ the interval
$[T_1, T_2]$ has a deep glen, and so we have a decomposition of $[T_1, T_2]$ into subintervals
\begin{equation}\label{eq:13.17}
  [T_1, T_2] = [T_1, T_{12} [ \ds \cup [T_{12}, T_{21}] \ds \cup ] T_{21}, T_2]
\end{equation}
such that
\begin{equation}\label{eq:13.18}
  \begin{array}{cc}
  [T_1, T_2] \cap \st_{Y_1} (A) & = [T_1, T_{12} [ \; , \\[2mm]
  [T_1, T_2] \cap \st_{Y_2} (B) & = \ ] T_{21}, T_{2} ],
  \end{array}
  \end{equation}
with
\begin{equation}\label{eq:13.19}
  M_W(T_1 T_2) = M_W [T_{12}, T_{21}) \notin \Yup,
  \end{equation}
\begin{equation}\label{eq:13.20}
\CS(W,T_{12}) = \CS(W,Z) = \CS(W,T_{21}).
  \end{equation}
  This subdivision can be deduced from the defining formula \eqref{eq:0.10} of a CS-ratio and the formulas for $M_W(T_1,T_2)$ in \S\ref{sec:6} in the case that $\CS(W,-)$ is not monotone on $[T_1,T_2]$, and the  formulas of the glen of $[T_1, T_2]$ in  \S\ref{sec:12}. These formulas show that square roots suffice for the above subdivision.  We omit the details.

To store the facts \eqref{eq:13.17}--\eqref{eq:13.20}, we say, that the set of all path components of $Z$-equivalence classes $ \neq C_\emptyset$ is the \textbf{$W$-$Z$-archipelago} in $\Ray(V)$ (for given rays $W$ and $Z$ in $V$ with $\CS(W,Z) > 0$), and that these path components are the \textbf{$W$-$Z$-islands} in $\Ray(V)$,  having proved that $\Zup$ is the disjoint union of all $W$-$Z$-islands and the set $\{ X \in \Zup \ds | \chX = \emptyset \} $, and that, for any two intervals $A,B$ and rays $T_1 \in A$, $T_2 \in B$, the interval $[T_1, T_2]$ has glen $[T_{12}, T_{21}]$ in the ``deep sea''
$$ \Ray(V ) \setminus \Zup  = \{ X \in \Ray(V) \ds | \CS(W,X) \leq \CS(W,Z)\} $$
while $[T_1, T_{12} [ \subset A$, $]T_{21}, T_2] \subset B$.

A further study is needed to describe the sets of $W$-$Z$-islands which constitute the $Z$-equivalence classes in $\Zup$ different from the useless class of non-polar rays. This study is left for a future work.

\end{document}